\documentclass[sn-mathphys-num]{sn-jnl}

\usepackage{graphicx}%
\usepackage{multirow}%
\usepackage{amsmath,amssymb,amsfonts}%
\usepackage{amsthm}%
\usepackage{mathrsfs}%
\usepackage[title]{appendix}%
\usepackage{xcolor}%
\usepackage{textcomp}%
\usepackage{manyfoot}%
\usepackage{booktabs}%
\usepackage{algorithm}%
\usepackage{algorithmicx}%
\usepackage{algpseudocode}%
\usepackage{listings}%
\usepackage{natbib}
\usepackage{subfigure}  

\def \argm{{\rm argmin} }

\theoremstyle{thmstyleone}%
\newtheorem{theorem}{Theorem}
\newtheorem{proposition}[theorem]{Proposition}%

\theoremstyle{thmstyletwo}%
\newtheorem{example}{Example}%
\newtheorem{remark}{Remark}%
\newtheorem{lemma}{Lemma}%

\theoremstyle{thmstylethree}%
\newtheorem{definition}{Definition}%

\numberwithin{theorem}{section}
\numberwithin{lemma}{section}
\numberwithin{corollary}{section}

\begin{document}

\title[Extrapolated Hard Thresholding Algorithms with Finite Length for Composite $\ell_0$ Penalized Problems]{Extrapolated Hard Thresholding Algorithms with Finite Length for Composite $\ell_0$ Penalized Problems}


\author[1]{\fnm{Fan} \sur{Wu}}\email{wufanmath@163.com}

\author[1]{\fnm{Jiazhen} \sur{Wei}}\email{jiazhenwei98@163.com}

\author*[1]{\fnm{Wei} \sur{Bian}}\email{bianweilvse520@163.com}

\affil[1]{\orgdiv{School of Mathematics}, \orgname{Harbin Institute of Technology}, \orgaddress{\city{Harbin},  \country{China}}}




\abstract{For a class of sparse optimization problems with the penalty function of $\|(\cdot)_+\|_0$, we first characterize its local minimizers and then propose an extrapolated hard thresholding algorithm to solve such problems. We show that the iterates generated by the proposed algorithm with $\epsilon>0$ (where $\epsilon$ is the dry friction coefficient) have finite length, without relying on the Kurdyka-{\L}ojasiewicz inequality. Furthermore, we demonstrate that the algorithm converges to an $\epsilon$-local minimizer of this problem. For the special case that $\epsilon=0$, we establish that any accumulation point of the iterates is a local minimizer of the problem. Additionally, we analyze the convergence when an error term is present in the algorithm, showing that the algorithm still converges in the same manner as before, provided that the errors asymptotically approach zero. Finally, we conduct numerical experiments to verify the theoretical results of the proposed algorithm.}




\maketitle

\section{Introduction}\label{intro}
In this paper, we consider composite $\ell_0$ penalized optimization problems of the form
\begin{equation}\label{pro1}
\min_{x\in\Omega}\ F(x):= f(x)+\lambda_1\|x_+\|_0+\lambda_2\|x_-\|_0,
\end{equation}
where $f:\mathbb{R}^n\to\mathbb{R}$ is a convex function whose gradient is $L_f$-Lipschitz continuous, $\lambda_1>0$ and $\lambda_2\geq0$ are weight coefficients. For any vector $x\in\mathbb{R}^n$, $\|x_+\|_0:=\sum_{i=1}^n \mathbf{1}_{(0,+\infty)}(x_i)$ and $\|x_-\|_0:=\sum_{i=1}^n \mathbf{1}_{(0,+\infty)}(-x_i)$, where 
$\mathbf{1}_{(0,+\infty)}(t):=\left\{
\begin{aligned}
&1,\  t>0 \\
&0,\ t\leq 0 
\end{aligned}
\right.$ 
is called the Heaviside function. 
Obviously, $\|x_+\|_0$ represents the number of all positive elements of $x$. The constraint set is defined as
\begin{equation*}
\Omega=\{x\in\mathbb{R}^n:l_i\leq x_i\leq u_i,\ \forall i=1,2,\cdots,n\},
\end{equation*}
where $-\infty\leq l_i \leq 0$, $0\leq u_i \leq +\infty$ with $l_i<u_i$ for all $i=1,2,\cdots,n$. 
It is easy to observe that $\|x\|_0=\|x_+\|_0+\|x_-\|_0$ for any $x\in\mathbb{R}^n$. This implies that problem \eqref{pro1} with $\lambda_1=\lambda_2$ reduces to the following $\ell_0$ penalized sparse optimization problem 
\begin{equation}\label{pro2}
\min_{x\in\Omega}\ F_0(x):= f(x)+\lambda_1\|x\|_0.
\end{equation}

The optimization problems involving the Heaviside function frequently arise in various practical applications, such as image and signal processing \cite{Candes2006Robust}, machine learning \cite{DY2018Hu} and binary classification problem \cite{WANG2022SU}. However, problem \eqref{pro1} is generally NP-hard due to the nonconvexity and discontinuity of the Heaviside function. This challenge motivates us to develop an extrapolated hard thresholding algorithm for addressing composite $\ell_0$ penalized optimization problem \eqref{pro1}, leveraging the special structure of the Heaviside function and the dry-friction property to ensure a finite length trajectory.

At present, there are mainly two classes of methods to solve Heaviside (or $\ell_0$) penalized problems: one is based on continuous relaxations or smoothing approximations of Heaviside functions \cite{bian2024nonsmooth,Cui2023The,Ermoliev1995The}, and the other is to directly optimize the Heaviside function \cite{Han2024ana,Lu2014Iterative,Zhang20224Zero}. Notably, the $\ell_1$ norm, as the optimal continuous convex relaxation of the $\ell_0$ norm, can efficiently identify sparse solutions and has a wide range of applications. However, using the $\ell_1$ norm sometimes leads to over-penalization or biased estimates. Then various continuous nonconvex relaxation penalty functions of the $\ell_0$ norm are proposed, such as the smoothly clipped absolute deviation (SCAD) function \cite{Fan2001Var}, the hard thresholding function \cite{Zheng2014High} and the capped-$\ell_1$ function \cite{Peleg2008A}. Some of these relaxation functions can be rewritten as the difference of two convex functions (DC), which has also garnered significant attention \cite{Bian2020a,Gasso2009Re,Thi2008A}. In this work, we adopt the direct method to solve problem \eqref{pro1}. 
The proximal operator of the $\ell_0$ norm has a closed-form solution due to its special structure. Utilizing this property, Blumensath and Davies \cite{Blumensath2008Iterative,Blumensath2009Iterative} proposed an iterative hard thresholding (IHT) algorithm to solve $\ell_0$ penalized least squares problems and proved that it converges to a local minimizer under conditions on the linear operator $A\in\mathbb{R}^{m\times n}$. 
The IHT algorithm is particularly well-suited for large-scale optimization problems and demonstrates strong performance.
With the advancement of science and technology and increasing practical demands, sparse optimization models have evolved from compressed sensing models to more general optimization frameworks.
Inspired by the IHT algorithm and the block coordinate descent (BCD) method, Lu and Zhang \cite{Lu2012Sparse} developed a penalty decomposition algorithm for solving general $\ell_0$-penalized and $\ell_0$-constrained minimization problems. Subsequently, Lu \cite{Lu2014Iterative} independently proposed an IHT algorithm and its variants to address problem \eqref{pro2}, showing convergence to a local minimizer. In recent years, algorithmic research on problem \eqref{pro1} with $\lambda_2=0$ has received considerable attention \cite{WANG2022SU,Zhang20224Zero,Zhou2021Quadratic}, particularly when the Heaviside function is composed with a linear operator. In \cite{WANG2022SU,Zhang20224Zero}, the authors reformulated the problem as an equality-constrained optimization problem and developed $L_{0/1}$ ADMM and inexact augmented {L}agrangian methods to solve it. Zhou et al. \cite{Zhou2021Quadratic} addressed the case that $f$ is twice continuously differentiable. Notably, all these works emphasized that the proximal operator of the Heaviside function has a closed-form solution. However, to the best of our knowledge, no prior work has specifically employed the IHT algorithm to solve the discontinuous composite $\ell_0$ penalized optimization problem \eqref{pro1}. This paper seeks to fill that gap.

The Kurdyka-{\L}ojasiewicz (K{\L}) property is a key tool for establishing the convergence of the iterates of continuous or discrete dynamical systems.
Here we review several first-order algorithms satisfying the finite length property for nonconvex problems, which have a similar condition or structure to that of problem \eqref{pro1}. For minimizing a proper and lower semicontinuous (l.s.c.) function in \cite{Attouch2009On}, Attouch and Bolte used the K{\L} property to establish the global convergence and convergence rate of the proximal algorithm, and showed that the limit point is a critical point. 
Subsequently, reference \cite{Attouch2013Co} analyzed the finite length property of bounded sequences generated by inexact versions of several algorithms, including the proximal algorithm \cite{Attouch2009On}, for different nonconvex closed objective functions satisfying K{\L} property.
Li and Pong \cite{Li2018Ca} studied the proximal gradient algorithm for minimizing the sum of a smooth function with Lipschitz continuous gradient and a proper closed function under the K{\L} property. They proved that the algorithm converges locally linearly to a critical point and further has a finite length. 
For the same model in \cite{Li2018Ca}, Jia et al. \cite{Jia2023Convergence} relaxed the gradient condition of the loss function to local Lipschitz continuous and showed that this together with the K{\L} property is sufficient to obtain the global convergence and rate-of-convergence results of the proximal gradient algorithm.
In \cite{Nguyen2024Fast}, the authors introduced a dry-like friction property in a nonconvex setting and designed a fast gradient algorithm for minimizing a gradient Lipschitz continuous function. They proved that the algorithm converges to a critical point when the K{\L} property holds.
For additional results on length finite within the framework of the K{\L} property, readers can refer to \cite{Attouch2013PR,Bot2019A,Bolte2007The,Laszlo2021Co} and the references therein.
It is important to note that all the works mentioned above rely on the K{\L} property to show that the trajectory has a finite length and the limit point is only a critical point.

Without relying on the K{\L} property, Adly and Attouch \cite{Adly2020finite} proposed a new class of proximal gradient algorithms with finite length to minimize a function whose gradient is Lipschitz continuous, which can be viewed as discrete time versions of the inertial gradient system incorporating dry friction and Hessian-driven damping.
It was proven that the limit point of the iterates is an approximate critical point. Additionally, they extended their work in \cite{Adly2022First} by considering the case without Hessian-driven damping, achieving similar results.
Although the algorithms proposed in \cite{Adly2020finite,Adly2022First} exhibit the finite length property and strong numerical performance, they are not directly applicable to problem \eqref{pro1} because of the nonsmoothness of the Heaviside function. This naturally raises a question of whether we can develop a first-order method with the finite length property for the composite $\ell_0$ penalized optimization problem \eqref{pro1} by leveraging the dry friction and the Hessian-driven damping.

Motivated by the works in \cite{Adly2020finite,Adly2022First} and the special structure of the Heaviside function, this paper is devoted to designing an extrapolated hard thresholding algorithm to solve problem \eqref{pro1}, achieving the finite length property without the aid of the K{\L} property. To the best of our knowledge, the finite length property of IHT algorithm for solving problem \eqref{pro1} given in this paper is new even in the context of $f$ convex setting. The analysis technique presented in this paper differs significantly from \cite{Adly2020finite,Adly2022First} due to the nonsmoothness and discontinuity of Heaviside function.
A key challenge is that the subdifferential of function $\|(\cdot)_+\|_0$ at the zero component is the half space, which complicates the application of first-order optimality conditions in certain scenarios.

The main contributions of this paper can be summarized as follows. (i) We provide equivalent characterizations of the local minimizers of problem \eqref{pro1} in different cases, and propose an extrapolated hard thresholding algorithm. (ii) For $\epsilon>0$, the sequence $\{x^k\}$ generated by the proposed algorithm has a finite length, i.e. $\sum_{k=1}^{\infty}\|x^{k+1}-x^k\|<\infty$. 
Notably, this is achieved without using the K{\L} property, distinguishing our approach from the literatures mentioned before. (iii) We show that the iterates converge to an $\epsilon$-local minimizer of problem \eqref{pro1}. Furthermore, when $\epsilon=0$, the algorithm exhibits the subsequential convergence to a local minimizer of problem \eqref{pro1}. (iv) We examine the stability of the proposed algorithm by introducing perturbations and errors, showing its robustness.

The paper is organized as follows. In section 2, we review fundamental definitions and preliminary results that will be used in the subsequent analysis. In section 3, we characterize the local minimizers of problem \eqref{pro1}. In section 4, we propose an extrapolated hard thresholding algorithm incorporating Hessian-driven damping and dry friction, and we establish its convergence properties. 
In section 5, we analyze the effect of the proposed algorithm with perturbations and the corresponding errors. Meanwhile, we examine its stability. In section 6, we present numerical experiments to demonstrate the good performance of the proposed algorithm in comparison to existing methods.

\section{Notations and preliminaries}
In what follows, we introduce some notations and recall a few definitions and properties concerning subdifferential, which will be used in the subsequent analysis.

Throughout the paper, we denote $\mathbb{N}=\{0,1,2,\cdots\}$ and $[n]=\{1,2,\cdots,n\}$. The $n$-dimensional Euclidean space is denoted by $\mathbb{R}^n$, endowed with the scalar product $\langle \cdot,\cdot \rangle$. $\|\cdot\|$ represents the Euclidean norm. For a nonempty closed convex set $\mathcal{Y}\subseteq\mathbb{R}^n$ and a point $x\in\mathcal{Y}$, $\mathcal{N}_{\mathcal{Y}}(x)$ means the normal cone of $\mathcal{Y}$ at $x$.
Given $x\in\mathbb{R}^n$ and $\delta>0$, we define $B(x, \delta)=\{y\in\mathbb{R}^n: \|y-x\|\leq \delta\}$. For a matrix $A\in\mathbb{R}^{m\times n}$, let $\|A\|$ denote its spectral norm, which is the square root of the largest eigenvalue of $A^{\mathrm{T}}A$. For any $x\in\mathbb{R}^n$, we define
$$\Gamma(x)=\{i : x_i\neq 0\}.$$
$\Gamma^c(x)\subseteq [n]$ represents the complement of $\Gamma(x)$. $|\Gamma(x)|$ denotes the number of the elements belonging to set $\Gamma(x)$.

\begin{definition}\cite{rockafellar2009variational}
	Consider a function $f:\mathbb{R}^n\to[-\infty,+\infty]$ and a point $\bar{x}$ with finite $f(\bar{x})$.
	The general (or limiting) subdifferential of $f$ at $\bar{x}$ is defined as
	\begin{equation*}
	\partial f(\bar{x})=\left\{ v\in\mathbb{R}^n: \exists\ x^k \xrightarrow[f]{} x\ \mathrm{and}\ v^k\in \widehat{\partial}f(x^k)\ \mathrm{with}\ v^k\to v \ \mathrm{as}\  k\to\infty \right\},
	\end{equation*}
	where $x^k \xrightarrow[f]{} x$ means $x^k\to x$ and $f(x^k)\to f(x)$, and $\widehat{\partial}f$ represents the regular subdifferential of $f$.
\end{definition}

The function $\max\{\cdot, 0\}:\mathbb{R}^n\to\mathbb{R}^n_+$ is a convex function. Therefore, it is continous, which together with l.s.c. of $\ell_0$ norm means that $\|x_+\|_0=\|\max\{x, 0\}\|_0$ is a l.s.c. function. Then we recall the form of the limiting subdifferential of function $\|(\cdot)_+\|_0$ and some conclusions on the limiting subdifferential.

\begin{lemma}\cite{Zhou2021Quadratic}
	For any fixed $\tilde{x}\in\mathbb{R}^n$, the limiting subdifferential of function $\|(\cdot)_+\|_0:\mathbb{R}^n\to\mathbb{R}$ at $\tilde{x}$ is
	\begin{equation}\label{L0P}
	\begin{aligned}
	\partial \|\tilde{x}_+\|_0
	=\left\{v\in\mathbb{R}^n:v_i=0,\ \mathrm{if}\ i\in\Gamma(\tilde{x});\ v_i\geq 0,\ \mathrm{if}\ i\in\Gamma^c(\tilde{x})\right\}.
	\end{aligned}
	\end{equation}
\end{lemma}

\begin{lemma}\cite{rockafellar2009variational}
	Suppose $f(x)=g(F(x))$ for a proper and l.s.c. function $g:\mathbb{R}^m\to [-\infty, +\infty]$ and a smooth mapping $F:\mathbb{R}^n\to\mathbb{R}^m$. Let $\bar{x}\in\mathbb{R}^n$ be
	a point where $f(\bar{x})$ is finite and the Jacobian $\nabla F(\bar{x})$ is of rank $m$. Then $\partial f(\bar{x})=\nabla F(\bar{x})^{\mathrm{T}}\partial g(F(\bar{x}))$.
\end{lemma}

	\vspace{0.12cm}

\begin{lemma}\cite{rockafellar2009variational}
	If $f(x)=f_1(x)+f_2(x)$, where $f_1$ is
	strictly continuous at $\bar{x}$ while $f_2$ is l.s.c. and proper with finite $f_2(\bar{x})$, then $\partial f(\bar{x})\subseteq \partial f_1(\bar{x})+ \partial f_2(\bar{x})$.
\end{lemma}

\section{Characteristics of local minimizers to problem \eqref{pro1}}

To facilitate the subsequent analysis, we provide the characterization of local minimizers of problem \eqref{pro1} under different cases.

\begin{proposition}\label{prop1}
	When $\lambda_2>0$ in problem \eqref{pro1}, for a vector $x^*\in\Omega$, the following two statements are equivalent.
	\begin{itemize}
		\item[$(i)$] $x^*$ is a local minimizer of problem \eqref{pro1}.
		\item[$(ii)$] $x^*$ is a global minimizer of function $f$ on set $\Omega_{\Gamma^c(x^*)}$, where $\Omega_{\Gamma^c(x^*)}=\{x\in \Omega: x_i=0,\ \forall i\in \Gamma^c(x^*)\}.$
	\end{itemize}
\end{proposition}
\begin{proof}
For any given $x\in\mathbb{R}^n$, we define the notations
\begin{equation*}
	\Gamma^{+}(x)=\{i : x_i > 0\}\quad\mbox{and}\quad \Gamma^{-}(x)=\{i : x_i < 0\}.
	\end{equation*}
	There exists a $\delta>0$ such that for any $x\in B(x^*,\delta)$, it holds that
	\begin{equation}\label{susp}
	\Gamma^{+}(x^*)\subseteq \Gamma^{+}(x)\quad\mbox{and}\quad \Gamma^{-}(x^*)\subseteq \Gamma^{-}(x).
	\end{equation}
Therefore, for any $x\in\Omega_{\Gamma^{c}(x^*)}\cap B(x^*,\delta)$, we have
	\begin{equation}\label{pne}
	\lambda_1\|x^*_+\|_0+\lambda_2\|x^*_-\|_0
	=\lambda_1\|x_+\|_0+\lambda_2\|x_-\|_0.
	\end{equation}

$(i)$ $\Rightarrow$ $(ii)$ If $x^*$ is a local minimizer of problem \eqref{pro1}, there exists a $\bar{\delta}\in (0, \delta]$ such that for any $x\in \Omega\cap B(x^*,\bar{\delta})$,
\begin{equation*}
\begin{aligned}
f(x^*)+\lambda_1\|x^*_+\|_0+\lambda_2\|x^*_-\|_0 
\leq f(x)+\lambda_1\|x_+\|_0+\lambda_2\|x_-\|_0.
\end{aligned}
\end{equation*}
This together with \eqref{pne} means
$f(x^*)\leq f(x)$ for $\forall x\in \Omega_{\Gamma^c(x^*)}\cap B(x^*,\bar{\delta})$,
which implies that result $(ii)$ holds due to the convexity of function $f$ and set $\Omega_{\Gamma^c(x^*)}$.
	
\vspace{0.12cm}

$(ii)$ $\Rightarrow$ $(i)$ If $x^*$ is a global minimizer of function $f$ on set $\Omega_{\Gamma^c(x^*)}$, for any $x\in\Omega_{\Gamma^c(x^*)}\cap B(x^*,\delta)$, it follows that
\begin{equation*}
f(x^*)+\lambda_1\|x^*_+\|_0+\lambda_2\|x^*_-\|_0
\leq f(x)+\lambda_1\|x_+\|_0+\lambda_2\|x_-\|_0,
\end{equation*}
which holds due to \eqref{pne}.

From the continuity of function $f$, we know that there is a $\hat{\delta}\in (0, \delta]$ such that for any $x\in\Omega\cap B(x^*,\hat{\delta})$, it holds that
\begin{equation}\label{so1}
f(x^*)\leq f(x)+\min\{\lambda_1,\lambda_2\}.
\end{equation}
For any $x\in(\Omega\setminus\Omega_{\Gamma^c(x^*)})\cap B(x^*,\hat{\delta})$, in view of \eqref{susp}, we obtain 
\begin{equation}\label{so2}
\lambda_1\|x^*_+\|_0+\lambda_2\|x^*_-\|_0+\min\{\lambda_1, \lambda_2\}\leq \lambda_1\|x_+\|_0+\lambda_2\|x_-\|_0.
\end{equation}
Combininig \eqref{so1} and \eqref{so2}, for any $x\in(\Omega\setminus\Omega_{\Gamma^c(x^*)})\cap B(x^*,\hat{\delta})$, it yields that
\begin{equation*}
\begin{aligned}
f(x^*)+\lambda_1\|x^*_+\|_0+\lambda_2\|x^*_-\|_0 
\leq f(x)+\lambda_1\|x_+\|_0+\lambda_2\|x_-\|_0.
\end{aligned}
\end{equation*}
\end{proof}

According to the result of Proposition \ref{prop1}, if $\lambda_2>0$, it holds that
\begin{equation*}\label{opf}
\begin{aligned}
x^*\ \mbox{is\ a\ local\ minimizer\ of\ problem\ \eqref{pro1}}
 &\Leftrightarrow 
x^*\in\mathop{\argm}\limits_{x\in \Omega_{\Gamma^c(x^*)}} f(x) \\
&\Leftrightarrow 0\in[\nabla f(x^*)+\mathcal{N}_{\Omega}(x^*)]_{\Gamma(x^*)}.
\end{aligned}
\end{equation*}
By the definition of set $\Omega$, we have
$\mathcal{N}_{\Omega}(x^*)=\mathcal{N}_{\Omega_1}(x_1^*)\times\mathcal{N}_{\Omega_2}(x_2^*)\times\cdots\times\mathcal{N}_{\Omega_n}(x_n^*),$ where $\Omega_i=\{x_i\in\mathbb{R}:l_i\leq x_i\leq u_i\}$ for any $i\in[n]$. Hence, when $\lambda_2>0$, we have
\begin{equation*}
x^*\ \mbox{is\ a\ local\ minimizer\ of\ problem\ \eqref{pro1}}
\Longleftrightarrow 0\in[\nabla f(x^*)]_i +\mathcal{N}_{\Omega_i}(x_i^*),\  \forall i\in\Gamma(x^*).
\end{equation*}
Based on this fact, we introduce the definition of the $\epsilon$-local minimizer for problem \eqref{pro1} with  $\lambda_2>0$.

\begin{definition}\label{def2}
	For any $\epsilon>0$, we say that $x^*\in\Omega$ is an $\epsilon$-local minimizer of problem \eqref{pro1} with  $\lambda_2>0$ if it satisfies
	\begin{equation}\label{de2}
	0\in[\nabla f(x^*)]_i +\mathcal{N}_{\Omega_i}(x^*_i)+[-\epsilon, \epsilon],\quad \forall i\in\Gamma(x^*).
	\end{equation}
\end{definition}

If $\epsilon=0$ in \eqref{de2}, $x^*$ obviously is a local minimizer of problem \eqref{pro1} with $\lambda_2>0$.

\begin{proposition}\label{lamd2}
	When $\lambda_2=0$ in problem \eqref{pro1}, for a vector $x^*\in\Omega$, the following two statements are equivalent.
	\begin{itemize}
		\item[$(i)$] $x^*$ is a local minimizer of problem \eqref{pro1}.
		
		\item[$(ii)$] $x^*$ is a global minimizer of function $f$ on set $\Omega_{s^*}$, where 
		\begin{equation}\label{os}
		\Omega_{s^*}=\{x\in \Omega: x_i\leq 0,\ \forall i\in \Gamma^c(x^*)\}.
		\end{equation}
	\end{itemize}
\end{proposition}

\begin{proof}
	When $\lambda_2=0$, problem \eqref{pro1} is reduced to the following form
	\begin{equation*}
	\min_{x\in\Omega} f(x)+\lambda_1\|x_+\|_0.
	\end{equation*}
	There exists a constant $\delta>0$ such that \eqref{susp} holds for any $x\in B(x^*,\delta)$.
	Therefore, for any $x\in\Omega_{s^*}\cap B(x^*,\delta)$, we have
	\begin{equation}\label{2pne}
	\lambda_1\|x^*_+\|_0=\lambda_1\|x_+\|_0.
	\end{equation}
	
	$(i)$ $\Rightarrow$ $(ii)$ If $x^*$ is a local minimizer of problem \eqref{pro1}, there is a $\bar{\delta}\in (0, \delta]$ satisfying
	\begin{equation*}
	\begin{aligned}
	f(x^*)+\lambda_1\|x^*_+\|_0
	\leq f(x)+\lambda_1\|x_+\|_0,\quad\forall x\in \Omega\cap B(x^*,\bar{\delta}).
	\end{aligned}
	\end{equation*}
	This together with \eqref{2pne} means that $f(x^*)\leq f(x)$ for any $x\in \Omega_{s^*}\cap B(x^*,\bar{\delta}).$
	
	\vspace{0.12cm}
	
	$(ii)$ $\Rightarrow$ $(i)$ If $x^*$ is a global minimizer of function $f$ on set $\Omega_{s^*}$, we have $f(x^*)\leq f(x)$ for any $x\in \Omega_{s^*}$. Then in view of \eqref{2pne}, it follows that
	\begin{equation*}\label{sfx}
	f(x^*)+\lambda_1\|x^*_+\|_0
	\leq f(x)+\lambda_1\|x_+\|_0,\quad \forall x\in\Omega_{s^*}\cap B(x^*,\delta).
	\end{equation*}
	
	Moreover, by the continuity of function $f$, there exists a $\hat{\delta}\in (0, \delta]$ such that
	\begin{equation}\label{2so1}
	f(x^*)\leq f(x)+\lambda_1,\quad\forall x\in\Omega\cap B(x^*,\hat{\delta}).
	\end{equation}
	Based on \eqref{susp}, for any $x\in(\Omega\setminus\Omega_{s^*})\cap B(x^*,\hat{\delta})$, we observe that 
	$\lambda_1\|x^*_+\|_0+\lambda_1\leq \lambda_1\|x_+\|_0,$ which together with \eqref{2so1} implies that
	\begin{equation*}
	\begin{aligned}
	f(x^*)+\lambda_1\|x^*_+\|_0 
	\leq f(x)+\lambda_1+\lambda_1\|x^*_+\|_0
	\leq f(x)+\lambda_1\|x_+\|_0.
	\end{aligned}
	\end{equation*}
	Then, $x^*$ is a local minimizer of problem \eqref{pro1}.
\end{proof}

\begin{remark}
	According to the proof of Proposition \ref{lamd2}, it is obvious that the conclusion in Proposition \ref{lamd2} also holds for $\Omega^{s^*}:=\{x\in \Omega: x_i\leq 0,\ \mbox{if}\ x_i^*\leq 0\}$.
\end{remark}

By the definition of $\Omega_{s^*}$ in \eqref{os}, we have $\mathcal{N}_{\Omega_{s^*}}(x^*)=\mathcal{N}_{[{\Omega_{s^*}}]_1}(x_1^*)\times\cdots\times\mathcal{N}_{[{\Omega_{s^*}}]_n}(x_n^*)$, where $[{\Omega_{s^*}}]_i$ equals to $\Omega_i$ if $i\in \Gamma(x^*)$ and $[l_i, 0]$ otherwise.
Based on the form of the normal cone, for any $i\in [n]$, we have
\begin{equation*}\label{ns}
\mathcal{N}_{[{\Omega_{s^*}}]_i}(x_i^*)=\left\{
\begin{aligned}
&\mathbb{R},&\ &\mbox{if}\ x^*_i=0\ \mbox{and}\ l_i=0,& \\
&\mathbb{R}_+,&\ &\mbox{if}\ x^*_i=0\ \mbox{and}\ l_i<0,& \\
&\tilde{\mathcal{N}}_i,&\ &\mbox{if}\ x^*_i\neq 0,&
\end{aligned}\right.
\end{equation*}
where $\tilde{\mathcal{N}}_i:=\{y_i\in\mathbb{R}:\langle y_i, x_i-x_i^* \rangle\leq 0,\ \forall x_i\in [l_i, u_i] \}.$
Combining this and the equivalence relationship established in Proposition \ref{lamd2}, we define a class of $\epsilon$-local minimizers of problem \eqref{pro1} for the case that $\lambda_2=0$.

\begin{definition}\label{def3}
For any $\epsilon>0$, we say that $x^*\in\Omega$ is an $\epsilon$-local minimizer of problem \eqref{pro1} with  $\lambda_2=0$ if it satisfies
	\begin{equation*}\label{de3}
	0\in[\nabla f(x^*)]_i +\mathcal{N}_{[\Omega_{s^*}]_i}(x^*_i)+[-\epsilon, \epsilon],\quad \forall i\in\Gamma(x^*)\cup J_{l}^*,
	\end{equation*}
where $\Omega_{s^*}$ is defined as in \eqref{os} and $J_{l}^*= \{j\in [n]:x^*_j=0\ \mbox{and}\ l_j<0\}$.
\end{definition}

\section{Algorithm and its convergence analysis}\label{sec7}

Inspired by the works in \cite{Adly2020finite,Adly2022First}, we develop an extrapolated hard thresholding algorithm with dry friction and Hessian-driven damping to solve problem \eqref{pro1} by discretizing the following inertial gradient system:
\begin{equation}\label{in1}
\begin{aligned}
0\in \ddot{x}(t)&+\gamma\dot{x}(t)+\epsilon\partial\|\dot{x}(t)\|_1+\beta\nabla^2 f(x(t))\dot{x}(t) \\
&+\nabla f(x(t))+\lambda_1\partial \|x(t)_+\|_0+\lambda_2\partial \|x(t)_-\|_0,
\end{aligned}
\end{equation}
where $\gamma>0$ and $\beta\geq0$. $\dot{x}(t)$ represents the viscous damping, introduced by Polyak \cite{Polyak1964S} to accelerate the gradient method. The term $\epsilon\partial\|\dot{x}(t)\|_1$ models dry friction. The potential friction function $\epsilon\|\cdot\|_1$ with $\epsilon>0$ is a convex function and exhibits a sharp minimum at the origin, fulfilling the dry friction property. 
$\nabla^2 f(x(t))\dot{x}(t)$ in \eqref{in1} represents the geometric damping, which can be interpreted as the derivative of $\nabla f(x(t))$ with respect to time $t$. This explains the emergence of a first-order algorithm by the time discretization of the system.
The Hessian-driven damping plays a crucial role in neutralizing oscillations within the inertial system.

We discretize the differential inclusion \eqref{in1} implicitly with respect to the nonsmooth functions $\|\cdot\|_1$ and $\|(\cdot)_+\|_0$, and explicitly with respect to the smooth function $f$. Taking a fixed time step size $h>0$, we obtain
\begin{equation}\label{dis}
\begin{aligned}
0\in \frac{1}{h^2}&(x^{k+1}-2x^k+x^{k-1})+\frac{\gamma}{h}(x^{k+1}-x^k)+\frac{\beta}{h}\left(\nabla f(x^k)-\nabla f(x^{k-1})\right) \\
&+\nabla f(x^k) +\epsilon \partial \left\|\frac{1}{h}(x^{k+1}-x^k)\right\|_1+\lambda_1 \partial\|x^{k+1}_+\|_0+\lambda_2 \partial\|x^{k+1}_-\|_0.
\end{aligned}
\end{equation}
Solving \eqref{dis} with respect to $x^{k+1}$, we derive Algorithm \ref{alg1}.

To ensure that the iterates $\{x^k\}$ generated by the proposed algorithm remain within the constraint set $\Omega$, we define the following set
\begin{equation*}
\Omega^k=\left\{x\in\mathbb{R}^n: x_i\in\left[l^k_i,\ u^k_i\right], \forall i\in[n]\right\},
\end{equation*}
where $l^k_i=\frac{l_i-x_i^k}{h}$ and $u^k_i=\frac{u_i-x_i^k}{h}$ for all $i\in [n]$.
For convenience, we define the function
\begin{equation*}
\begin{aligned}
Q(x,z,y)=\frac{h\epsilon}{1+h\gamma} \|y\|_1&+\frac{\lambda_1}{1+h\gamma}\|(x+hy)_+\|_0 \\ &+\frac{\lambda_2}{1+h\gamma}\|(x+hy)_-\|_0+\frac{1}{2}\|y-z\|^2
\end{aligned}
\end{equation*}
and denote
\begin{equation}\label{Yk}
w^{k+1}=\frac{x^{k+1}-x^k}{h}.
\end{equation}

\begin{algorithm}
	\caption{Extrapolated Hard Thresholding Algorithm with Hessian-driven Damping and Dry Friction}\label{alg1}
	\begin{algorithmic}[1]
		\Require Take initial points $x^{-1}=x^0\in\Omega$. Choose parameters $\epsilon\geq0$, $\beta\geq0$, $h>0$ and $\gamma>0$ satisfying $\gamma\geq \frac{1}{h}+(2\beta+h)L_f$. Set $k=0$.
		\While{a termination criterion is not met,}
		\State \textbf{Step 1:} Compute
		\begin{align}
		&v^k=\frac{1}{h(1+h\gamma)}(x^k-x^{k-1})-\frac{\beta}{1+h\gamma}(\nabla f(x^k)-\nabla f(x^{k-1})), \label{al}\\
		&z^k=v^k-\frac{h}{1+h\gamma}\nabla f(x^k),\label{al2}\\ 
		&x^{k+1}\in x^k+h\cdot\arg\min_{y\in\Omega^k}Q(x^k,z^k,y).\label{al1}
		\end{align}
		\State \textbf{Step 2:} Increment $k$ by one and return to \textbf{Step 1}.
		\EndWhile
	\end{algorithmic}
\end{algorithm}

It is not difficult to observe that Algorithm \ref{alg1} shares a structure similar to classical proximal gradient algorithms. Specifically, it involves the proximal operator of the nonsmooth function and the gradient of the smooth function in the optimization problem.
However, there are two key differences between the two algorithms. First, the proximal mapping in Algorithm \ref{alg1} is defined for the sum of the function $\lambda\|(\cdot)_+\|_0$ and the dry friction function $\epsilon\|\cdot\|_1$. This reflects the implicitly discrete nature of the dry friction function, which plays a crucial role in ensuring the finite length of the sequence generated by the algorithm. Second, unlike the classical proximal gradient algorithm, which relies solely on the gradient information at the previous point $x^k$, the proposed algorithm incorporates the difference between gradients at consecutive iterates $x^k$ and $x^{k-1}$. This modification is achieved through the Hessian-driven damping term.

From \eqref{al1}, it is evident that for any $k\in\mathbb{N}$, the following inequality holds
\begin{equation}\label{op1}
Q(x^k,z^k,w^{k+1})\leq Q(x^k,z^k,y),\quad\forall y\in\Omega^k.
\end{equation}

\begin{lemma}\label{lema1}
	Denote $\{x^k\}$ the sequence generated by Algorithm \ref{alg1}, then for any $k\in \mathbb{N}$, we have
	\begin{equation}\label{eq1}
	\begin{aligned}
	&\lambda_1\|x^k_+\|_0+\lambda_2\|x^k_-\|_0 - \left[\lambda_1\|x_+^{k+1}\|_0+\lambda_2\|x^{k+1}_-\|_0\right] \\
	\geq&h\epsilon\|w^{k+1}\|_1 +\frac{1}{2}(h\gamma-1)\|w^{k+1}\|^2+\langle w^{k+1}, w^{k+1}-w^{k} \rangle \\
	&\quad+\beta\langle w^{k+1},\nabla f(x^k)-\nabla f(x^{k-1})\rangle+h\langle w^{k+1}, \nabla f(x^k) \rangle.
	\end{aligned}
	\end{equation}
\end{lemma}

\begin{proof}
	For any $k\in\mathbb{N}$, we have $0\in \Omega^k$ by $x^k\in\Omega$.
	Using \eqref{op1}, it holds that $Q(x^k,z^k,w^{k+1})\leq Q(x^k,z^k,0)$, namely
	\begin{equation}\label{feq}
	\begin{aligned}
	&\frac{\lambda_1}{1+h\gamma}\|x^k_+\|_0 +\frac{\lambda_2}{1+h\gamma}\|x^k_-\|_0+\frac{1}{2}\|z^k\|^2\\
	\geq &\frac{h\epsilon}{1+h\gamma} \|w^{k+1}\|_1+\frac{\lambda_1}{1+h\gamma}\|(x^k+hw^{k+1})_+\|_0+\frac{\lambda_2}{1+h\gamma}\|(x^k+hw^{k+1})_-\|_0  \\
	&\qquad+\frac{1}{2}\|w^{k+1}-z^k\|^2.
	\end{aligned}
	\end{equation}	 
	From \eqref{Yk}, we have $x^{k+1}=x^k+hw^{k+1}$, which together with \eqref{feq} implies	 
	\begin{equation}\label{wfi}
	\begin{aligned}
	&\frac{\lambda_1}{1+h\gamma}\|x^k_+\|_0 +\frac{\lambda_2}{1+h\gamma}\|x^k_-\|_0\\
	\geq&\frac{h\epsilon}{1+h\gamma} \|w^{k+1}\|_1+\frac{\lambda_1}{1+h\gamma}\|x^{k+1}_+\|_0
	+\frac{\lambda_2}{1+h\gamma}\|x^{k+1}_-\|_0 +\frac{1}{2}\|w^{k+1}\|^2-\langle w^{k+1}, z^k \rangle.
	\end{aligned}
	\end{equation}	 
	
	By the definitions of $z^k$ and $w^k$, it follows that
	\begin{equation}\label{ns2}
	\begin{aligned}
	\langle w^{k+1}, z^k\rangle=&\langle w^{k+1}, v^k-\frac{h}{1+h\gamma}\nabla f(x^k)\rangle \\
	=&\frac{1}{1+h\gamma}\langle w^{k+1}, w^{k} \rangle
	-\frac{h}{1+h\gamma}\langle w^{k+1}, \nabla f(x^k)\rangle \\
	&\quad -\frac{\beta}{1+h\gamma}\langle w^{k+1}, \nabla f(x^k)-\nabla f(x^{k-1}) \rangle.
	\end{aligned}
	\end{equation}
	Plugging \eqref{ns2} into \eqref{wfi}, and then multiplying it by $(1+h\gamma)$, we obtain
	\begin{equation}\label{ns4}
	\begin{aligned}
	&\lambda_1\|x^k_+\|_0 +\lambda_2\|x^k_-\|_0 \\
	\geq& h\epsilon \|w^{k+1}\|_1+\lambda_1\|x^{k+1}_+\|_0+\lambda_2\|x^{k+1}_-\|_0+\frac{1}{2}(1+h\gamma)\|w^{k+1}\|^2-\langle w^{k+1}, w^{k} \rangle  \\
	&\qquad+h\langle w^{k+1}, \nabla f(x^k)\rangle +\beta\langle w^{k+1}, \nabla f(x^k)-\nabla f(x^{k-1}) \rangle.
	\end{aligned}
	\end{equation}
	
	We note that
	\begin{equation}\label{se}
	\begin{aligned}
	\frac{1}{2}(1+h\gamma)\|w^{k+1}\|^2-\langle w^{k+1}, w^{k} \rangle 
	= \frac{1}{2}(h\gamma -1)\|w^{k+1}\|^2+\langle w^{k+1}, w^{k+1}- w^{k} \rangle,
	\end{aligned}
	\end{equation}
	which together with \eqref{ns4} derives \eqref{eq1}.
\end{proof}

Building upon the above lemma, we establish one of main conclusions of this paper. For simplicity, we define the sequence $\{E_k\}_{k\in\mathbb{N}}$ as follows:
\begin{equation}\label{ek}
E_k:=\frac{1}{2h}\beta L_f \|x^k-x^{k-1}\|^2 + \frac{1}{2}\left\| \frac{x^{k}-x^{k-1}}{h}\right\|^2 + F(x^k)-\inf_{\Omega}F,
\end{equation}
which acts as the energy function in the analysis.

\begin{theorem}\label{th1}
	Let $\{x^k\}$ be the sequence generated by Algorithm \ref{alg1} with $\epsilon>0$. Then, it holds that
	\begin{equation}\label{sumk}
	\sum_{k=0}^{\infty} \|x^{k+1}-x^k\| \leq\frac{1}{\epsilon}\left(F(x^0)-\inf_{\Omega}F\right)< +\infty.
	\end{equation}
\end{theorem}

\begin{proof}
	In view of the Lipschitz continuity of $\nabla f$ and the Cauchy-Schwarz inequality, we have
	\begin{equation*}\label{s4}
	\begin{aligned}
	f(x^{k+1})-f(x^k)-\frac{1}{2}L_f\|x^{k+1}-x^k\|^2 
	&\leq \left\langle \nabla f(x^k),\ x^{k+1}-x^k\right\rangle\\
	&=h\left\langle \nabla f(x^k),\ w^{k+1}\right\rangle
	\end{aligned}
	\end{equation*}
	and
	\begin{equation*}\label{s5}
	\begin{aligned}
	\left| \left\langle\nabla f(x^k)-\nabla f(x^{k-1}),\ w^{k+1}\right\rangle \right|
	& \leq \|\nabla f(x^k)-\nabla f(x^{k-1})\|\cdot\|w^{k+1}\| \\
	& \leq L_fh\|w^{k}\|\cdot\|w^{k+1}\|  \\
	&\leq \frac{1}{2}L_fh\left(\|w^{k}\|^2 + \|w^{k+1}\|^2\right).
	\end{aligned}
	\end{equation*}
	Moreover, it holds that
	\begin{equation*}\label{s3}
	\left\langle w^{k+1}-w^{k},\ w^{k+1}\right\rangle = \frac{1}{2}\|w^{k+1}\|^2 - \frac{1}{2}\|w^{k}\|^2 + \frac{1}{2}\|w^{k+1}-w^{k}\|^2.
	\end{equation*}
Collecting the above results and plugging them into \eqref{eq1}, it yields
	\begin{equation*}
	\begin{aligned}
	&\lambda_1\|x^k_+\|_0 +\lambda_2\|x^k_-\|_0-\left[\lambda_1\|x^{k+1}_+\|_0+\lambda_2\|x^{k+1}_-\|_0\right]  \\
	\geq&  h\epsilon\|w^{k+1}\|_1 +\frac{1}{2}(h\gamma-1)\|w^{k+1}\|^2+\frac{1}{2}\|w^{k+1}\|^2
	- \frac{1}{2}\|w^{k}\|^2 \\
	&\quad + \frac{1}{2}\|w^{k+1}-w^{k}\|^2- \frac{1}{2}\beta h L_f\|w^{k}\|^2
	- \frac{1}{2}\beta hL_f\|w^{k+1}\|^2 \\
	&\quad +f(x^{k+1})-f(x^k)-\frac{1}{2}L_f\|x^{k+1}-x^k\|^2,
	\end{aligned}
	\end{equation*}
	which means that
		\begin{equation}\label{eq22}
	\begin{aligned}
	&\lambda_1\|x^k_+\|_0 +\lambda_2\|x^k_-\|_0-\left[\lambda_1\|x^{k+1}_+\|_0+\lambda_2\|x^{k+1}_-\|_0\right]  \\
	\geq&\epsilon \|x^{k+1}-x^k\|_1 + \frac{1}{2}(\frac{\gamma}{h}-\frac{1}{h^2}-\frac{\beta }{h}L_f-L_f)\|x^{k+1}-x^k\|^2 \\
	&\quad+\frac{1}{2}\|w^{k+1}\|^2- \frac{1}{2}\|w^{k}\|^2+\frac{1}{2h^2}\|x^{k+1}-2x^k+x^{k-1}\|^2 \\
	&\quad -\frac{1}{2h}\beta L_f\|x^k-x^{k-1}\|^2+f(x^{k+1})-f(x^k).
	\end{aligned}
	\end{equation}
	Since $\gamma\geq\frac{1}{h}+(2\beta+h)L_f$, it follows that $\frac{1}{2}(\frac{\gamma}{h}-\frac{1}{h^2}-\frac{\beta}{h}L_f-L_f)\geq \frac{1}{2h}\beta L_f$.
	Using this inequality together with the definition of $E_k$ in \eqref{ek} and \eqref{eq22}, we obtain
	\begin{equation}\label{ek1}
	\epsilon \|x^{k+1}-x^k\|_1+\frac{1}{2h^2}\|x^{k+1}-2x^k+x^{k-1}\|^2+E_{k+1}-E_k\leq 0.
	\end{equation}
	Then we deduce that $E_{k+1}\leq E_k$, which together with $E_k\geq 0$ implies that $\lim_{k\to\infty} E_k$ exists. Based on \eqref{ek1}, we have
	\begin{equation*}
	\epsilon \|x^{k+1}-x^k\|\leq \epsilon \|x^{k+1}-x^k\|_1 \leq E_k-E_{k+1}.
	\end{equation*}
	After summation of this inequality with respect to $k\in\mathbb{N}$, then using $\epsilon>0$ and the existence of $\lim_{k\to\infty} E_k$, we obtain that \eqref{sumk} holds.
\end{proof}

\begin{theorem}\label{th2}
	Suppose that $\{x^k\}$ is the sequence generated by Algorithm \ref{alg1} with $\epsilon>0$. Then, there exists $x^{\infty}\in\Omega$ such that $\lim_{k\to\infty}x^k=x^{\infty}$
	and $x^{\infty}$ is an $\epsilon$-local minimizer of problem \eqref{pro1}.
\end{theorem}

\begin{proof}
	From \eqref{sumk}, we know that sequence $\{x^k\}$ is convergent. Then we denote its limit as $x^{\infty}$. 
	
	For any $i\in \Gamma(x^{\infty})$, in view of $\lim_{k\to\infty}x^k=x^{\infty}$, we know that there exists $N\in\mathbb{N}$ such that
	\begin{equation*}
	x_i^{k+1}\neq 0,\quad\forall k\geq N.
	\end{equation*}
	Using $x^{k+1}=x^k+hw^{k+1}$, \eqref{L0P} and \eqref{al1}, for any $i\in \Gamma(x^{\infty})$, we have
	\begin{equation}\label{sop1}
	\begin{aligned}
	0\in \frac{h\epsilon}{1+h\gamma}\partial |w_i^{k+1}| + w_i^{k+1}- v_i^k+ \frac{h}{1+h\gamma}[\nabla f(x^k)]_i +\mathcal{N}_{\Omega_i^k}(w_i^{k+1}),\quad \forall k\geq N.
	\end{aligned}
	\end{equation}
	
	Next, we first show
	\begin{equation}\label{omg}
	\mathcal{N}_{\Omega^k}(w^{k+1})=\mathcal{N}_{\Omega}(x^{k+1}).
	\end{equation}
	In fact, for any $x\in\Omega^k=\{x\in\mathbb{R}^n: \frac{l-x^k}{h}\leq x\leq \frac{u-x^k}{h}\}$, set $q=hx+x^k$, then we get
	\begin{equation*}
	x\in\Omega^k\Longleftrightarrow q\in\Omega.
	\end{equation*}
	Moreover, for any $z\in\mathcal{N}_{\Omega^k}(w^{k+1}),$  it holds that
	\begin{equation*}
	\begin{aligned}
	0\geq \langle z, x-w^{k+1} \rangle =&\left\langle z, x- \frac{x^{k+1}-x^k}{h} \right\rangle ,\quad\forall x\in\Omega^k\\
	=&\left\langle z, \frac{q-x^k}{h}- \frac{x^{k+1}-x^k}{h} \right\rangle, \quad\forall q\in\Omega \\
	=&\frac{1}{h}\left\langle z, q- x^{k+1}\right\rangle, \quad\forall q\in\Omega,
	\end{aligned}
	\end{equation*}
	which is equivalent to $z\in\mathcal{N}_{\Omega}(x^{k+1})$ as $h>0$.
	
	Multiplying \eqref{sop1} by $\frac{1+h\gamma}{h}$ and using \eqref{omg}, then for $\forall k\geq N$, it yields that
	\begin{equation*}\label{kopt}
	\begin{aligned}
	0\in\ &\epsilon\partial |w_i^{k+1}|  + \frac{1+h\gamma}{h}w_i^{k+1}- \frac{1+h\gamma}{h} v_i^k+ [\nabla f(x^k)]_i\\
	& +\mathcal{N}_{\Omega_i}(x_i^{k+1}),\quad\forall i\in\Gamma(x^{\infty}),
	\end{aligned}
	\end{equation*}
	which is equivalent to the fact that there is $\xi^{k+1}_i\in \partial |w^{k+1}_i|$ such that for any $x_i\in \Omega_i$,
	\begin{equation}\label{opl}
	\begin{aligned}
	\left\langle -\frac{1+h\gamma}{h} w_i^{k+1}-\epsilon\xi^{k+1}_i+ \frac{1+h\gamma}{h}v_i^k- [\nabla f(x^k)]_i,\ x_i-x_i^{k+1}\right\rangle\leq 0.
	\end{aligned}
	\end{equation}
	
	In view of \eqref{sumk}, we observe that 
	\begin{equation}\label{xwk0}
	\lim_{k\to\infty}\|x^k-x^{k-1}\|=0\quad\mbox{and}\quad \lim_{k\to\infty} w^{k+1}=0.
	\end{equation}
	Combining this and the Lipschitz continuity of $\nabla f$, we obtain
	\begin{equation}\label{nafi}
	\lim_{k\to\infty}\nabla f(x^k) =\nabla f(x^{\infty})
	\end{equation}
	and $\lim_{k\to\infty}[\nabla f(x^k)-\nabla f(x^{k-1})]=0$, and then it yields by \eqref{al} that
	\begin{equation}\label{vk}
	\lim_{k\to\infty} v^k=0.
	\end{equation}
	
	Using the upper semicontinuity of $\partial |\cdot|$, \eqref{xwk0}, \eqref{nafi} and \eqref{vk}, and letting $k\to\infty$ in \eqref{opl}, we obtain that for any $i\in\Gamma(x^{\infty})$, there exists $\bar{\xi}_i\in\partial |0|$ satisfying
	\begin{equation*}\label{epf}
	\langle -\epsilon \bar{\xi}_i-[\nabla f(x^{\infty})]_i, x_i-x_i^{\infty} \rangle\leq 0,\quad\forall x_i\in\Omega_i.
	\end{equation*}
We further have
		\begin{equation}\label{pf}
	0\in [\nabla f(x^{\infty})]_i+\mathcal{N}_{\Omega_i}(x_i^{\infty})+[-\epsilon, \epsilon],\quad\forall i\in \Gamma(x^{\infty}).
	\end{equation}
	Thus, \eqref{pf} means that $x^{\infty}$ is an $\epsilon$-local minimizer of problem \eqref{pro1} with $\lambda_2>0$.
		
For the case that $\lambda_2=0$, we first define
	$$\Omega_{s^{\infty}}=\{x\in \Omega: x_j\leq 0,\ \forall j\in \Gamma^c(x^\infty)\}$$
	and $J_{l}^{\infty}=\{j\in [n]:x^{\infty}_j=0\ \mbox{and}\ l_j<0\}.$
 Then for any $j\in J_{l}^{\infty}$, we have
	\begin{equation}\label{nse}
	\mathcal{N}_{\Omega_j}(x_j)+\mathcal{N}_{[l_j, 0]}(x_j)
	\subseteq\mathcal{N}_{[\Omega_{s^{\infty}}]_j}(x_j),\quad\forall x_j\in [\Omega_{s^{\infty}}]_j,
	\end{equation}
which holds because of $[\Omega_{s^{\infty}}]_j=\Omega_j\cap [l_j, 0]$ for any $j\in J_{l}^{\infty}$.
	Based on \eqref{al1}, it yields that
	$0\in\frac{h\epsilon}{1+h\gamma}\partial |w_j^{k+1}|+ \frac{\lambda_1h}{1+h\gamma}\partial\|(x^{k+1}_j)_+\|_0+w_j^{k+1}- z_j^k  +\mathcal{N}_{\Omega_j^k}(w_j^{k+1})$ for any $j\in [n]$,
	which is equivalent to
	\begin{equation}\label{op2}
	\begin{aligned}
	0\in &\frac{h\epsilon}{1+h\gamma}\partial |w_j^{k+1}|+\frac{\lambda_1h}{1+h\gamma}\partial\|(x^{k+1}_j)_+\|_0+w_j^{k+1}- z_j^k  +\mathcal{N}_{\Omega_j}(x_j^{k+1}),
	\end{aligned}
	\end{equation}
	where we use \eqref{omg}. 
	Moreover, we have
	\begin{equation}\label{zs1}
	\lim_{k\to\infty}z^{k}=-\frac{h}{1+h\gamma}\nabla f(x^{\infty}),
	\end{equation}
	which holds due to \eqref{nafi} and \eqref{vk}.
	By \eqref{L0P} and $\frac{\lambda_1h}{1+h\gamma}>0$, we observe that
	\begin{equation}\label{nop}
	\frac{\lambda_1h}{1+h\gamma}\partial\|(x^{k+1}_j)_+\|_0=\partial\|(x^{k+1}_j)_+\|_0,\quad\forall j\in [n].
	\end{equation} 
	
	For any $j\in J_{l}^{\infty}$, we know that $x^{k+1}_j\to 0$, which implies that there exist a subsequence of $\{x^{k+1}_j\}$ (for simplicity we also denote it as $\{x^{k+1}_j\}$) and a sufficiently large $N\in\mathbb{N}$ such that for all $k\geq N$, it holds that
	\begin{equation*}
	0=x_j^{k+1}\to 0\quad \mbox{or}\quad 0\neq x_j^{k+1}\to 0\ \mbox{and}\  x_j^{k+1}\in (l_j, u_j).
	\end{equation*}
	Then we consider the following two cases.
	
	$\textbf{(a)}$ When $0=x_j^{k+1}\to 0$ for all $k\geq N$, using \eqref{op2} and \eqref{nop}, there is $a^{k+1}\geq 0$ satisfying $0\in \frac{h\epsilon}{1+h\gamma}\partial |w_j^{k+1}|+a^{k+1}+w_j^{k+1}- z_j^k  +\mathcal{N}_{\Omega_j}(x^{\infty}_j).$ Therefore, we have
	\begin{equation}\label{k01}
	0\in \frac{h\epsilon}{1+h\gamma}\partial |w_j^{k+1}|+w_j^{k+1}- z_j^k  +\mathcal{N}_{[\Omega_{s^{\infty}}]_j}(x^{\infty}_j),
	\end{equation}
	which holds due to \eqref{nse} and $\mathcal{N}_{[l_j, 0]}(0)=[0, +\infty)$ with $l_j<0$. Letting $k\to \infty$ in \eqref{k01} and using \eqref{xwk0} and \eqref{zs1}, we obtain
	\begin{equation}\label{aop}
	0\in [\nabla f(x^{\infty})]_j+\mathcal{N}_{[\Omega_{s^{\infty}}]_j}(x_j^{\infty})+[-\epsilon, \epsilon].
	\end{equation}
	
	$\textbf{(b)}$ When $0\neq x_j^{k+1}\to 0$ and $x_j^{k+1}\in (l_j, u_j)$ for all $k\geq N$, since $\partial\|(x^{k+1}_j)_+\|_0=0$, by \eqref{op2}, we observe that
	\begin{equation*}
	\begin{aligned}
	0\in \frac{h\epsilon}{1+h\gamma}\partial |w_j^{k+1}|+w_j^{k+1}- z_j^k  +\mathcal{N}_{\Omega_j}(x_j^{k+1}).
	\end{aligned}
	\end{equation*}
	We know that $\mathcal{N}_{\Omega_j}(x_j^{k+1})=0$ if $x_j^{k+1}\in (l_j, u_j)$. 
	Hence, it follows that
	\begin{equation}\label{2cs}
	\begin{aligned}
	0\in \frac{h\epsilon}{1+h\gamma}\partial |w_j^{k+1}|+w_j^{k+1}- z_j^k,\quad\forall k\geq N.
	\end{aligned}
	\end{equation}
Making $k\to\infty$ in \eqref{2cs}, we obtain 
	$0\in[-\epsilon, \epsilon]+[\nabla f(x^{\infty})]_j$, 
	which implies that \eqref{aop} holds.
	
	For any $i\in\Gamma(x^{\infty})$, it holds that $\mathcal{N}_{[{\Omega_{s^{\infty}}}]_i}(x_i^{\infty})=\mathcal{N}_{\Omega_i}(x_i^{\infty})$. This together with \eqref{pf} means
	\begin{equation*}\label{2epf}
	0\in [\nabla f(x^{\infty})]_i+\mathcal{N}_{[{\Omega_{s^{\infty}}}]_i}(x_i^{\infty})+[-\epsilon, \epsilon],\quad\forall i\in \Gamma(x^{\infty}).
	\end{equation*}
	Then we conclude that $x^{\infty}$ is an $\epsilon$-local minimizer of problem \eqref{pro1} with $\lambda_2=0$.
\end{proof}

Next, we discuss the convergence behavior of Algorithm \ref{alg1} when the dry friction coefficient $\epsilon$ is zero.

\begin{theorem}\label{th43}
	Assume that $f$ is level bounded on $\Omega$ and the parameters in Algorithm \ref{alg1} satisfy $\gamma>\frac{1}{h}+(2\beta+h)L_f$. Let $\{x^k\}$ be a sequence generated by Algorithm \ref{alg1} with $\epsilon=0$. Then,
	\begin{itemize}
		\item [$(i)$] there exists $M>0$ such that $\|x^k\|\leq M$, $\forall k\in\mathbb{N}$, and
		\begin{equation}\label{sk}
		\sum_{k=0}^{\infty}\|x^{k+1}-x^k\|^2<+\infty;
		\end{equation}
		
		\item [$(ii)$] any accumulation point of $\{x^k\}$ is a local minimizer of problem \eqref{pro1}.
	\end{itemize}
\end{theorem}

\begin{proof}
	$(i)$\  Following the procedure used to derive \eqref{ek1}, we observe that when $\epsilon=0$, the sequence $\{E_k\}$ remains non-increasing. This implies
	\begin{equation}\label{F0}
	F(x^k)-\inf_{\Omega} F\leq E_k\leq E_0=F(x^0)-\inf_{\Omega}F,\quad\forall k\in\mathbb{N}.
	\end{equation}
	By the nonnegativity of $\|(\cdot)_+\|_0$, it follows that $f(x^k)\leq F(x^k),$
	which together with \eqref{F0} and the level bounded property of $f$ on set $\Omega$ means that sequence $\{x^k\}$ is bounded.
	
	On the basis of \eqref{eq22}, $\epsilon=0$ and the definition of $E_k$, we obtain
	\begin{equation}\label{the2}
	\begin{aligned}
	0\geq \frac{1}{2}&(\frac{\gamma}{h}-\frac{1}{h^2}-\frac{2\beta}{h}L_f-L_f)\|x^{k+1}-x^k\|^2\\
	&+\frac{1}{2h^2}\|x^{k+1}-2x^k+x^{k-1}\|^2 +E_{k+1}-E_k,
	\end{aligned}
	\end{equation}
	which means that
	$\sum_{k=0}^{\infty} \frac{1}{2}(\frac{\gamma}{h}-\frac{1}{h^2}-\frac{2\beta}{h}L_f-L_f)\|x^{k+1}-x^k\|^2\leq E_0<+\infty.$
	Then using the condition $\gamma>\frac{1}{h}+(2\beta+h)L_f$, it holds that
	$\sum_{k=0}^{\infty}\|x^{k+1}-x^k\|^2<+\infty$.
	
	\vspace{0.12cm}
	
	$(ii)$\ 
	From \eqref{sk}, we notice that 
	\begin{equation}\label{axk}
	\lim_{k\to\infty}\|x^{k+1}-x^k\|=0\quad\mbox{and}\quad \lim_{k\to\infty} w^k=0.
	\end{equation}
	Suppose $x^*$ is an accumulation point of sequence $\{x^k\}$ with the convergence of subsequence $\{x^{k_i}\}$. 
	Using \eqref{axk}, we obtain $\lim_{i\to\infty}x^{k_i+1}=x^*$. Additionally, by the Lipschitz continuity of $\nabla f$, it follows that
	\begin{equation}\label{avk}
	\lim_{k\to\infty} v^k=0\quad\mbox{and}\quad\lim_{i\to\infty}\nabla f(x^{k_i})=\nabla f(x^*).
	\end{equation}
	
	For all $j\in\Gamma(x^*)$, there exists a sufficiently large $N\in\mathbb{N}$ such that
	\begin{equation}\label{ne1}
	x_j^{k_i+1}\neq 0,\quad\forall i\geq N.
	\end{equation}
	From the definition of $w^{k_i+1}$, we observe $x^{k_i+1}=x^{k_i}+hw^{k_i+1}$ and
	\begin{equation}\label{op}
	\begin{aligned}
	0\in &\frac{h\lambda_1}{1+h\gamma}\partial \|(x^{k_i}+hw^{k_i+1})_+\|_0-\frac{h\lambda_2}{1+h\gamma}\partial \|(x^{k_i} +hw^{k_i+1})_-\|_0 \\
	&\quad+w^{k_i+1}-z^{k_i}
	+\mathcal{N}_{\Omega^{k_i}}(w^{k_i+1}).
	\end{aligned}
	\end{equation}
	Combining \eqref{L0P}, \eqref{omg}, \eqref{ne1} and \eqref{op}, we obtain that for any $i\geq N$, it holds that
	\begin{equation}\label{er}
	\begin{aligned}
	0\in w_j^{k_i+1}-v_j^{k_i}
	+\frac{h}{1+h\gamma}[\nabla f(x^{k_i})]_j +\mathcal{N}_{\Omega_j}(x_j^{k_i+1}),\quad\forall j\in\Gamma(x^*).
	\end{aligned}
	\end{equation}
	Taking the limit in \eqref{er} with respect to $i$, and using \eqref{axk} and \eqref{avk}, it gives
	\begin{equation}\label{ops}
	0\in[\nabla f(x^*)+\mathcal{N}_{\Omega}(x^*)]_{\Gamma(x^*)}.
	\end{equation}
This indicates that $x^*$ is a global minimizer of function $f$ on set $\Omega_{\Gamma^c(x^*)}$. Then by Proposition \ref{prop1}, we obtain that $x^*$ is a local minimizer of problem \eqref{pro1} with $\lambda_2>0$.
	
Consider the case that $\lambda_2=0$. At each iteration, $w^{k+1}$ is obtained by solving the following subproblem
	\begin{equation*}\label{bk}
	w^{k+1}\in \arg\min_{y\in\Omega^k}\left\{\frac{\lambda_1}{1+h\gamma}\|(x^k+hy)_+\|_0
	+\frac{1}{2}\|y-z^k\|^2\right\}.
	\end{equation*}
	
	For a subsequence $\{x^{k_i}\}$ of $\{x^k\}$ with $x^{k_i}\to x^*$, by \eqref{omg}, it follows that
	\begin{equation*}
	0\in \frac{\lambda_1 h}{1+h\gamma}\partial\|x^{k_i+1}_+\|_0+w^{k_i+1}-z^{k_i}+\mathcal{N}_{\Omega}(x^{k_i+1}).
	\end{equation*}
	Furthermore, from \eqref{avk}, we have $\lim_{i\to\infty}z^{k_i}=-\frac{h}{1+h\gamma}\nabla f(x^*)$.
	Then, using arguments similar to those in Theorem \ref{th2}, we derive 
	\begin{equation*}\label{ops2}
	0\in [\nabla f(x^*)]_j+\mathcal{N}_{[\Omega_{s^*}]_j}(x^*_j),\quad\forall j\in J_l^*,
	\end{equation*}
	where $\Omega_{s^*}=\{x\in \Omega: x_i\leq 0,\ \forall i\in \Gamma^c(x^*)\}$ and $J_l^*= \{j\in [n]:x^*_j=0\ \mbox{and}\ l_j<0\}$. 
	Combining this result with \eqref{ops}, it follows that
	$x^*$ is a local minimizer of problem \eqref{pro1} with $\lambda_2=0$.
\end{proof}

\section{Errors and perturbations}
In this section, we consider Algorithm \ref{alg1} with errors and perturbations, which is derived by discretizing the following dynamic system
\begin{equation*}
\begin{aligned}
e(t)\in \ddot{x}(t)&+\gamma\dot{x}(t)+\epsilon\partial\|\dot{x}(t)\|_1+\beta\nabla^2 f(x(t))\dot{x}(t) \\
&+\nabla f(x(t))+\lambda_1\partial \|x(t)_+\|_0+\lambda_2\partial \|x(t)_-\|_0,
\end{aligned}
\end{equation*}
where $e(\cdot)$ represents errors and perturbations.
By taking a fixed time step size $h>0$ and applying a similar temporal discretization as in \eqref{dis}, we obtain
\begin{equation}\label{eis}
\begin{aligned}
e^k\in \frac{1}{h^2}&(x^{k+1}-2x^k+x^{k-1})+\frac{\gamma}{h}(x^{k+1}-x^k) \\
&+\nabla f(x^k) 
+\frac{\beta}{h}\left(\nabla f(x^k)-\nabla f(x^{k-1})\right) \\
&+\epsilon \partial \left\|\frac{1}{h}(x^{k+1}-x^k)\right\|_1+\lambda_1 \partial\|x^{k+1}_+\|_0+\lambda_2 \partial\|x^{k+1}_-\|_0,
\end{aligned}
\end{equation}

\begin{algorithm}
	\caption{Perturbation Extrapolated Hard Thresholding Algorithm with Hessian-driven Damping and Dry Friction}\label{alg3}
	\begin{algorithmic}[1]
		\Require Take initial points $x^{-1}=x^0\in\Omega$. Choose parameters $\epsilon\geq0$, $\beta\geq0$, $h>0$ and $\gamma>0$ such that $\gamma\geq \frac{1}{h}+(2\beta+h)L_f$. Set $k=0$.
		\While{a termination criterion is not met,}
		\State \textbf{Step 1:} Compute
		\begin{align}
		&v^k=\frac{1}{h(1+h\gamma)}(x^k-x^{k-1})-\frac{\beta}{1+h\gamma}(\nabla f(x^k)-\nabla f(x^{k-1})), \label{ev1}\\ 
		&z^k=v^k-\frac{h}{1+h\gamma}(\nabla f(x^k)-e^k),\label{el2}\\ 
		&x^{k+1}\in x^k+h\cdot\arg\min_{y\in\Omega^k}Q(x^k,z^k,y).\label{el1}
		\end{align}
		\State \textbf{Step 2:} Increment $k$ by one and return to \textbf{Step 1}.
		\EndWhile
	\end{algorithmic}
\end{algorithm}

\noindent where $e^k\in\mathbb{R}^n$ represents an unknown error term.
Solving \eqref{eis} with respect to $x^{k+1}$, then we derive the perturbed version of Algorithm \ref{alg1}, which is stated in Algorithm \ref{alg3}. The convergence analysis of Algorithm \ref{alg3} is given in Theorem \ref{ethe1} and Theorem \ref{ethe2}.

	\vspace{0.12cm}

\begin{theorem}\label{ethe1}
	Suppose that $\{x^k\}$ is the sequence generated by Algorithm \ref{alg3} with $\epsilon>0$ and the error sequence $\{e^k\}$ satisfies $\lim_{k\to\infty} e^k=0$. Then, it holds that
	\begin{equation*}\label{eumk}
	\sum_{k=0}^{\infty} \|x^{k+1}-x^k\| < \infty.
	\end{equation*}
	And there exists $x^{\infty}\in\Omega$ such that $\lim_{k\to\infty}x^k=x^{\infty}$. Furthermore, $x^{\infty}$ is an $\epsilon$-local minimizer of problem \eqref{pro1}.
\end{theorem}

\begin{proof}
	Based on \eqref{el2}, we have
	\begin{equation}\label{es2}
	\begin{aligned}
	\langle w^{k+1}, z^k\rangle=&\langle w^{k+1}, v^k-\frac{h}{1+h\gamma}(\nabla f(x^k)-e^k)\rangle \\
	=&\frac{1}{1+h\gamma}\langle w^{k+1}, w^{k} \rangle
	-\frac{h}{1+h\gamma}\langle w^{k+1}, \nabla f(x^k)\rangle\\
	&-\frac{\beta}{1+h\gamma}\langle w^{k+1}, \nabla f(x^k)-\nabla f(x^{k-1}) \rangle +\frac{h}{1+h\gamma} \langle w^{k+1}, e^k \rangle. \\
	\end{aligned}
	\end{equation}
	By the Cauchy-Schwarz inequality, it follows that $\langle w^{k+1}, e^k \rangle\leq \|w^{k+1}\| \cdot \|e^k\|$. This together with  \eqref{wfi} and \eqref{es2} implies that
	\begin{equation*}
	\begin{aligned}
	&\lambda_1\|x^k_+\|_0 +\lambda_2\|x^k_-\|_0 \\
	\geq&h\epsilon \|w^{k+1}\|_1+\lambda_1\|x^{k+1}_+\|_0+\lambda_2\|x^{k+1}_-\|_0+\frac{1}{2}(1+h\gamma)\|w^{k+1}\|^2-\langle w^{k+1}, w^{k} \rangle  \\
	&\quad
	+h\langle w^{k+1}, \nabla f(x^k)\rangle  +\beta\langle w^{k+1}, \nabla f(x^k)-\nabla f(x^{k-1}) \rangle-h\|w^{k+1}\| \cdot \|e^k\|.
	\end{aligned}
	\end{equation*}
	Combining this and \eqref{se}, we obtain
	\begin{equation*}\label{ens5}
	\begin{aligned}
	&\lambda_1\|x^k_+\|_0 +\lambda_2\|x^k_-\|_0-\left[\lambda_1\|x^{k+1}_+\|_0+\lambda_2\|x^{k+1}_-\|_0\right]\\
	\geq & h\epsilon \|w^{k+1}\|_1 + \frac{1}{2}(h\gamma -1)\|w^{k+1}\|^2+\left\langle w^{k+1}, w^{k+1}- w^{k} \right\rangle -h\|w^{k+1}\| \cdot \|e^k\| \\
	&\  +h\left\langle w^{k+1}, \nabla f(x^k)\right\rangle +\beta\left\langle w^{k+1}, \nabla f(x^k)-\nabla f(x^{k-1}) \right\rangle.\\
	\end{aligned}
	\end{equation*}
	Using a similar analysis to the beginning of Theorem \ref{th1}, it holds that
	\begin{equation}\label{k1}
	\left(\epsilon-\|e^k\|\right)\cdot \|x^{k+1}-x^k\| +E_{k+1}-E_k\leq 0,
	\end{equation}
	where $\|\cdot\|\leq \|\cdot\|_1$ is used. By $\lim_{k\to\infty} e^k=0$, we know that there exists a sufficiently large $N\in\mathbb{N}$ such that
	$\|e^k\|\leq \frac{\epsilon}{2}$ for all $k\geq N$, which together with \eqref{k1} yields that
	\begin{equation}\label{k2}
	\frac{\epsilon}{2} \|x^{k+1}-x^k\| +E_{k+1}-E_k\leq 0,\quad\forall k\geq N.
	\end{equation}
	For any $p> 0$, summing \eqref{k2} with respect to $k$ from $N$ to $N+p$, we obtain
	\begin{equation*}
	\sum_{k=N}^{N+p}\frac{\epsilon}{2} \|x^{k+1}-x^k\|\leq E_N-E_{N+p+1}\leq E_N.
	\end{equation*}
	Letting $p\to\infty$ in the above expression, we deduce that
	$\sum_{k=0}^{\infty} \|x^{k+1}-x^k\|<\infty$, which further implies the existence of $x^{\infty}$.
	
	Using the definitions in \eqref{ev1} and \eqref{el2}, along with $\lim_{k\to\infty} e^k=0$, we have $\lim_{k\to\infty}z^k=-\frac{h}{1+h\gamma}\nabla f(x^{\infty})$. The proof of the remaining results of this theorem follows a similar argument as in Theorem \ref{th2}, and thus, it is omitted here.
\end{proof}
	
\begin{theorem}\label{ethe2}
	Assume that the parameters in Algorithm \ref{alg3} satisfy $\gamma>\frac{1}{h}+(2\beta+h)L_f$. Let $\{x^k\}$ be a sequence generated by Algorithm \ref{alg3}, in which $\epsilon=0$ and $\sum_{k=0}^{\infty} \|e^k\|^2<\infty$. Then
		\begin{equation}\label{pr}
	\sum_{k=0}^{\infty}\|x^{k+1}-x^k\|^2<+\infty.
	\end{equation}  
	Moreover, for any $x^*\in\mathcal{X}:=\left\{\bar{x}\in \Omega: \exists \{x^{k_i}\}\subseteq \{x^k\}\  \mbox{satisfying}\ \lim_{i\to\infty}x^{k_i}=\bar{x}\right\}$, it is a local minimizer of problem \eqref{pro1}.
\end{theorem}

\begin{proof}
	The proof is similar to that of Theorem \ref{th43}, where the key for the analysis is 
	$\sum_{k=0}^{\infty}\|x^{k+1}-x^k\|^2<+\infty.$
	Then we only show how to obtain it. When $\epsilon =0$, using the same analysis as in \eqref{the2}, we have
	\begin{equation*}\label{esk}
	\begin{aligned}
	0\geq& \frac{1}{2}(\frac{\gamma}{h}-\frac{1}{h^2}-\frac{2\beta}{h}L_f-L_f)\|x^{k+1}-x^k\|^2-\|x^{k+1}-x^k\| \cdot \|e^k\|\\
	&+\frac{1}{2h^2}\|x^{k+1}-2x^k+x^{k-1}\|^2 +E_{k+1}-E_k.
	\end{aligned}
	\end{equation*}
	For any $J\in \mathbb{N}$, summing up the above inequality for $k=0,\cdots,J$, and then using $E_k\geq 0$ for all $k\in\mathbb{N}$, we obtain
	\begin{equation*}
\begin{aligned}
	\sum_{k=0}^{J}\zeta\|x^{k+1}-x^k\|^2
	&\leq E_0+\sum_{k=0}^{J}\|x^{k+1}-x^k\| \cdot \|e^k\| \\
	&\leq E_0+\sum_{k=0}^{J}\frac{\zeta}{2}\|x^{k+1}-x^k\|^2+\sum_{k=0}^{J}\frac{1}{2\zeta}\|e^k\|^2,
	\end{aligned}
	\end{equation*}
	where $\zeta:=\frac{1}{2}(\frac{\gamma}{h}-\frac{1}{h^2}-\frac{2\beta}{h}L_f-L_f)>0$ due to $\gamma>\frac{1}{h}+(2\beta+h)L_f$. Then it follows that
	\begin{equation*}
	\sum_{k=0}^{J}\frac{\zeta}{2}\|x^{k+1}-x^k\|^2\leq E_0+\sum_{k=0}^{J}\frac{1}{2\zeta}\|e^k\|^2.
	\end{equation*}
	Letting $J\to\infty$, we obtain $\sum_{k=0}^{\infty}\frac{\zeta}{2}\|x^{k+1}-x^k\|^2\leq E_0+ \sum_{k=0}^{\infty}\frac{1}{2\zeta}\|e^k\|^2$.
	Then we deduce that \eqref{pr} holds because of $\sum_{k=0}^{\infty} \|e^k\|^2<\infty$.
\end{proof}

\begin{remark}
	When $\beta=0$ in \eqref{al}, the convergence results established in Theorem \ref{th1} and Theorem \ref{th2} also hold
	if $\gamma\geq \frac{1}{h}+hL_f$. Further, the results in Theorem \ref{th43} hold for $\gamma> \frac{1}{h}+hL_f$.
	At this time, the algorithm can be obtained by the time discretization of the following inertial gradient system with dry friction:
	\begin{equation*}
0\in \ddot{x}(t)+\gamma\dot{x}(t)+\epsilon\partial\|\dot{x}(t)\|_1+\nabla f(x(t))+\lambda_1\partial \|x(t)_+\|_0+\lambda_2\partial \|x(t)_-\|_0.
	\end{equation*}
	Similar results also hold for the perturbed version of the algorithm.
\end{remark}

\section{Numerical experiments}

In this section, we present numerical results to demonstrate the effectiveness of the proposed algorithm. All experiments are conducted using Python 3.7.0 on a 64-bit Lenovo PC equipped with an Intel(R) Core(TM) i7-10710U CPU @1.10GHz (1.61GHz) and 16GB of RAM. The algorithms terminate when either $x^k$ satisfies the condition of being an $\bar{\epsilon}$-local minimizer, as defined in Definitions \ref{def2} and \ref {def3}, or the number of iterations exceeds $3000$.

For clarity in graphs and commentaries, we refer to the extrapolated hard thresholding algorithm with Hessian-driven damping and dry friction (Algorithm \ref{alg1}) as Alg1, and the extrapolated hard thresholding algorithm with Hessian-driven damping only (Algorithm \ref{alg1} with $\epsilon=0$) as $\mbox{Alg1}_{\epsilon=0}$. Additionally, ``Spar" represents the sparsity level of the true solution used to generate the problem data. 

\vspace{0.12cm}

\begin{example}\label{exa:5.1} 
We consider the following optimization problem in $\mathbb{R}^3$ to illustrate the effect of introducing the item $\nabla f(x^k)-\nabla f(x^{k-1})$ in the proposed algorithm:
\begin{equation}\label{be1}
\min_{\mathbf{-50}\leq x\leq \mathbf{50}} 500x^{\mathrm{T}}x+30x_2(x_1+x_3)+0.01 \|x\|_0.
\end{equation}
Obviously, the global minimizer of problem \eqref{be1} is $x^*=(0,0,0)$.
\end{example}

We first apply the accelerated iterative hard thresholding algorithm with the same extrapolation parameters as the well-known FISTA proposed by Beck and Teboulle in \cite{Beck2009A} ($\mbox{EHT}_{\footnotesize\mbox{\emph{FISTA}}}$) to solve problem \eqref{be1}. We then use Alg1 with $\beta$ set to $0$ and $0.005$ ($\mbox{Alg1}_{\beta=0.005}$) to solve the same problem. For this example, the initial point is chosen as $x^0=(20, 19, 20)$, and $L_f=\|\nabla^2 f(x)\|$. The parameters for $\mbox{EHT}_{\footnotesize\mbox{\emph{FISTA}}}$ are $L=2L_f$ and $t_1=1$, while the parameters for Alg1 are set as follows:
\begin{equation*}
\epsilon=1\mathrm{e}-3,\ \bar{\epsilon}=2\epsilon,\ h=0.1,\ \gamma=\frac{1}{h}+(2\beta+h)L_f.
\end{equation*}

\begin{figure}[htbp]
	\centerline{
		\subfigure[$\{f(x^k)\}$]{\label{fig1.a}\includegraphics[width=0.35\textwidth]{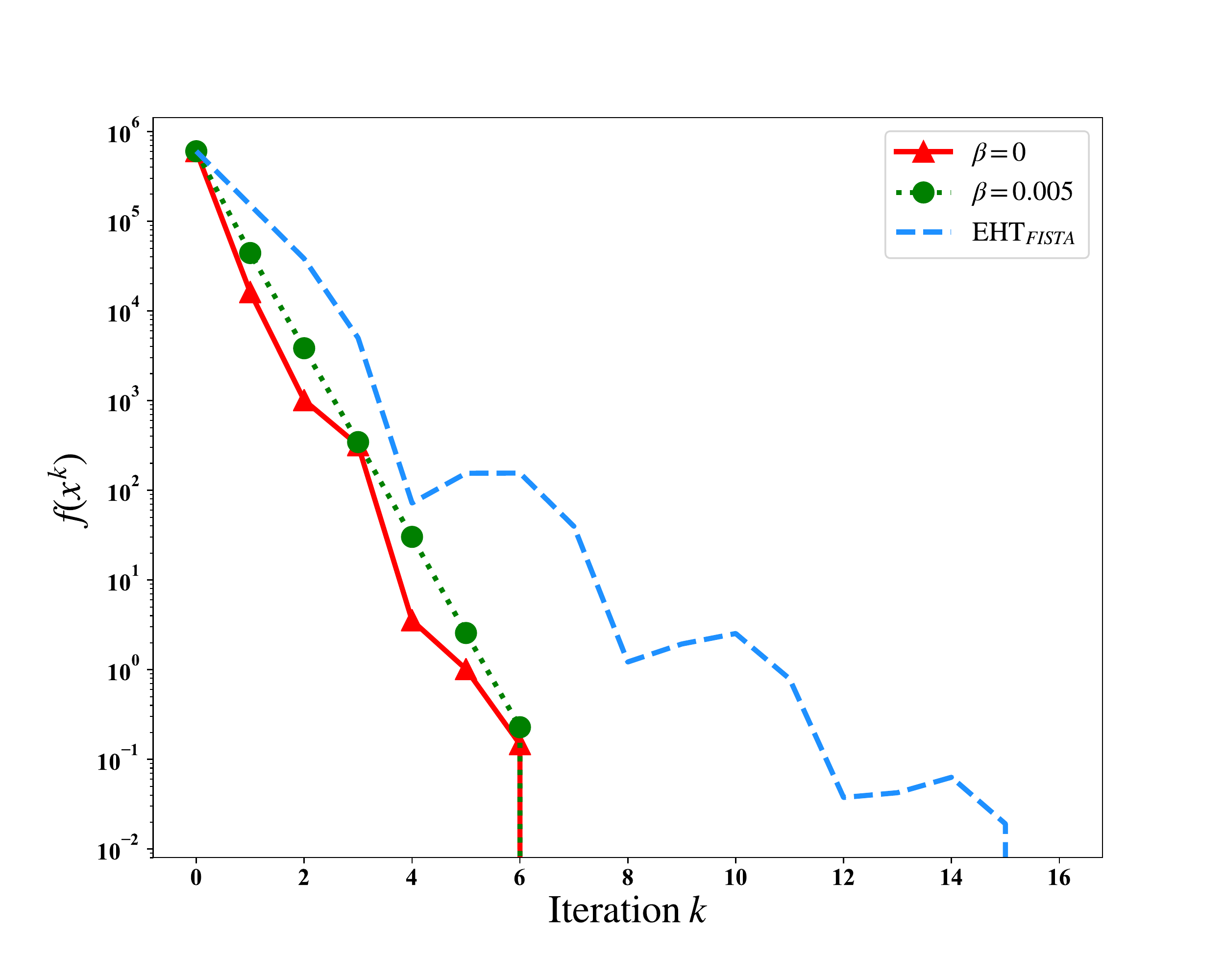}}
		\subfigure[$\{\|x^k\|_0\}$]{\label{fig1.b}\includegraphics[width=0.35\textwidth]{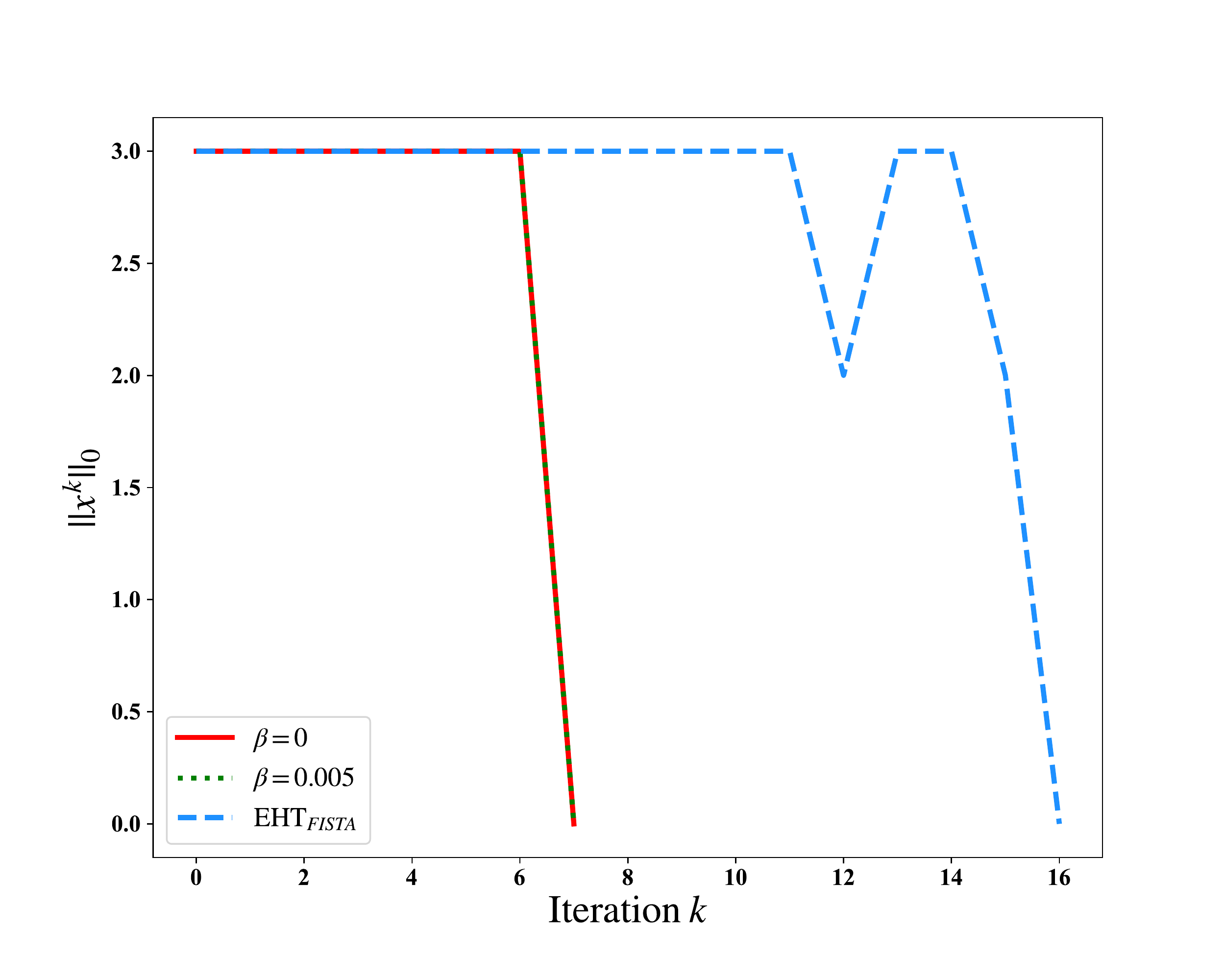}}
		\subfigure[$\{y^k_2\}$]{\label{fig1.c}\includegraphics[width=0.35\textwidth]{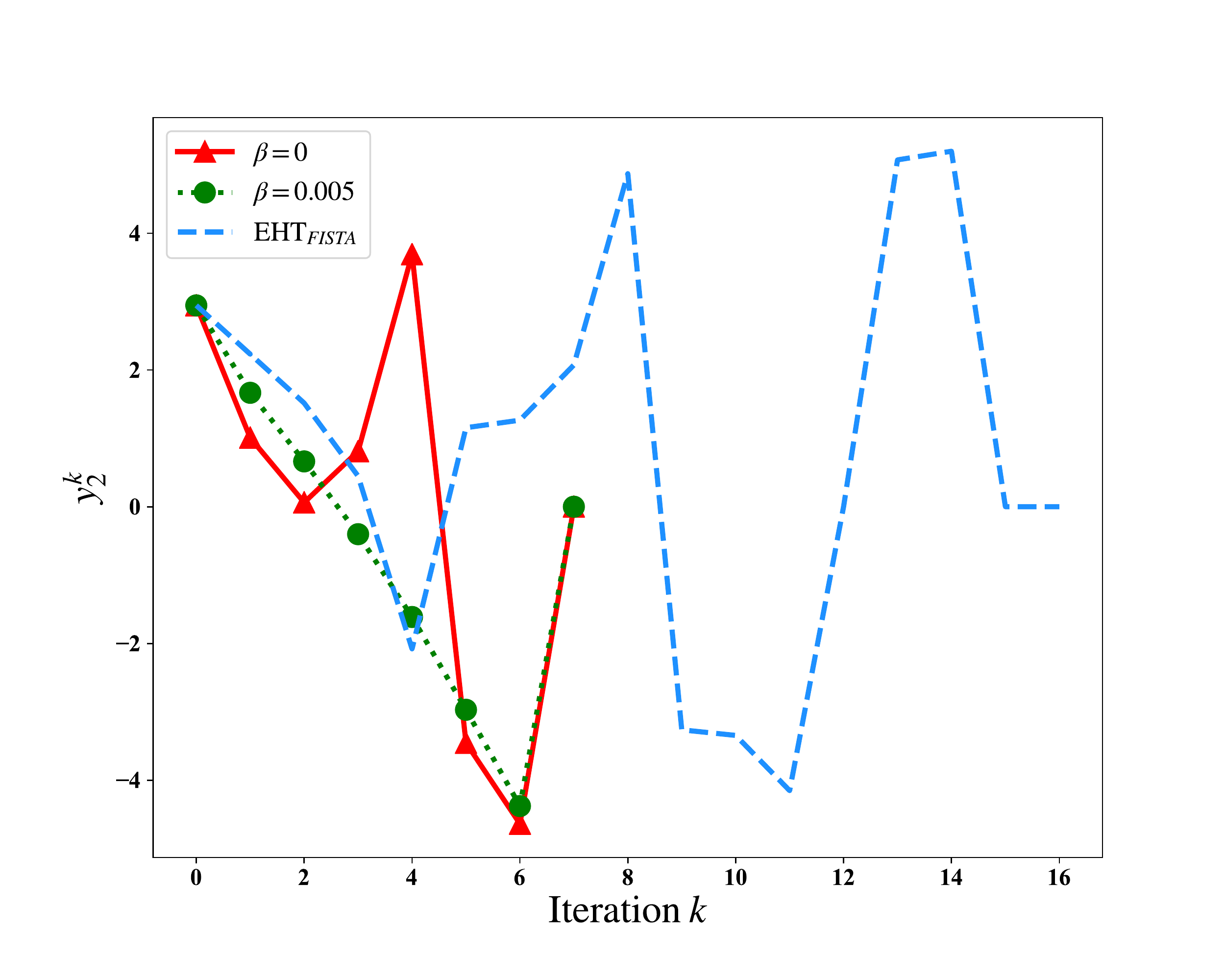}}
	}
	\caption{Numerical results for Example \ref{exa:5.1}.}
	\label{fig:fig5.1}       
\end{figure}

We compare the convergence of $\{f(x^k)\}$ and the evolution of $\{\|x^k\|_0\}$, as shown in Fig. \ref{fig:fig5.1}. The results demonstrate that Alg1 converges faster than $\mbox{EHT}_{\footnotesize\mbox{\emph{FISTA}}}$. Additionally, the numerical results in Fig. \ref{fig:fig5.1} reveal that the iterates produced by $\mbox{EHT}_{\footnotesize\mbox{\emph{FISTA}}}$ oscillate around $x^*$ before converging to it.
To better observe the behavior of $\{x^k\}$, we plot the sequence $\{y^k\}$ versus the iteration number $k$ in  Fig. \ref{fig1.c} and Fig. \ref{fig:fig6.2}, where $y_i^k:=\mathrm{sign}(x_i^k)*\mathrm{log}(|x_i^k|+\epsilon_1)$ with $\epsilon_1=1\mathrm{e}-14$, $\forall i\in [3]$. These figures show that the sequence $\{x^k\}$ generated by Alg1 with $\beta=0$ also exhibits oscillatory behavior, but it converges faster than $\mbox{EHT}_{\footnotesize\mbox{\emph{FISTA}}}$.
When $\beta=0.005$, the iterates generated by Alg1 converge directly to the global minimizer without oscillations, due to the presence of gradient differences between consecutive iterates.
In all cases, the iterates converge to the minimizer of problem \eqref{be1}, as illustrated in Fig. \ref{fig:fig5.1} and Fig. \ref{fig:fig6.2}.

\begin{figure}[htbp]
	\centerline{
		\includegraphics[width=0.6\textwidth]{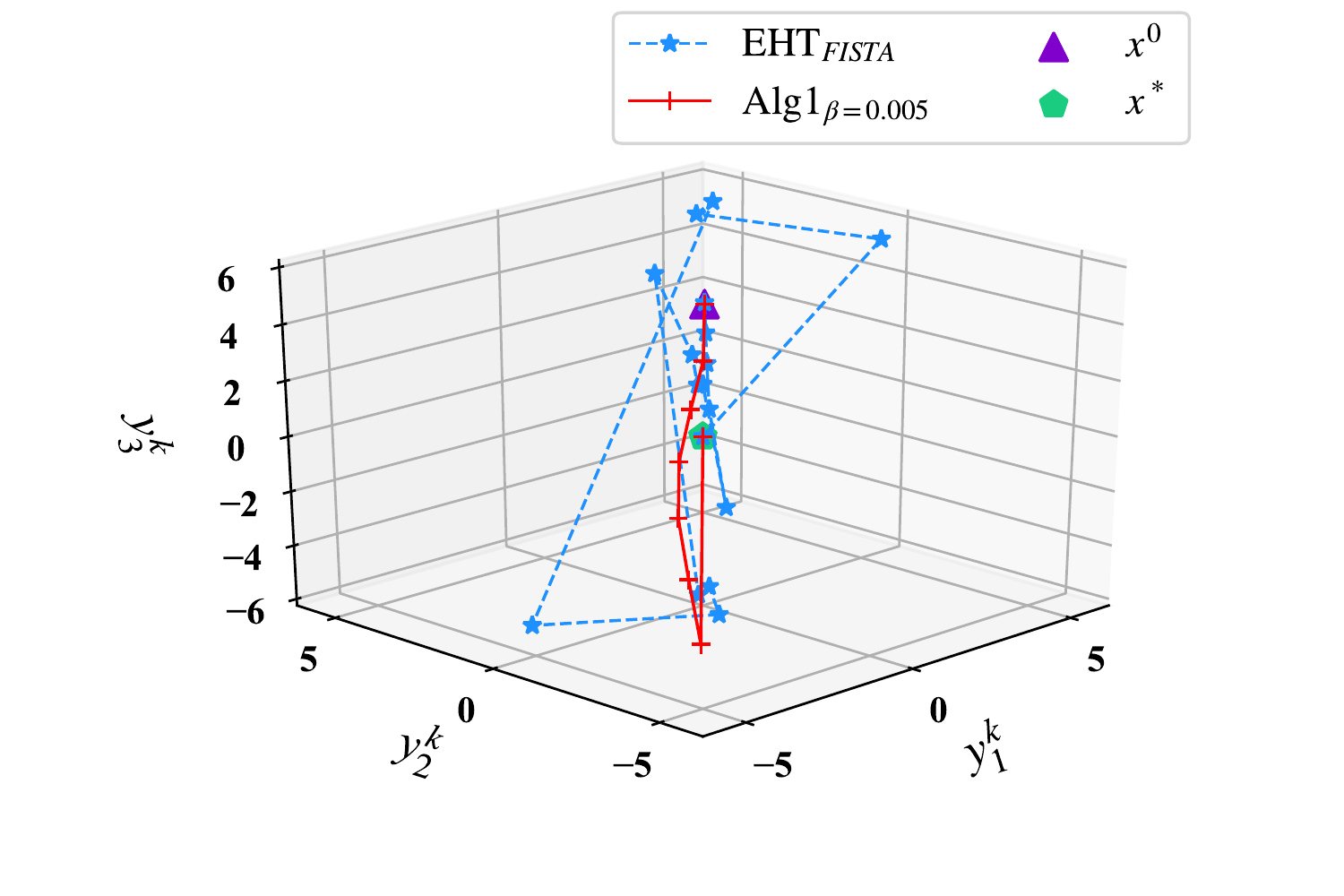}
	}
	\caption{The trajectories of $\{y^k\}$ generated by different algorithms for Example \ref{exa:5.1}.}
	\label{fig:fig6.2}       
\end{figure}


\begin{example}\label{ex2}
Linear regression is well-regarded for its extensive applications in areas such as signal and image processing and portfolio management, with the least squares loss function being the most commonly employed. In this example, we focus on a sparse problem of the following form:
\begin{equation}\label{cs}
\min_{\mathbf{-1}\leq x\leq \mathbf{5}} \frac{1}{2}\|Ax-b\|^2+0.06\|x\|_0,
\end{equation}
where $A\in\mathbb{R}^{m\times n}$ is the risk-return matrix and $b\in\mathbb{R}^m$ is the target return vector. We select $m<n$, which often appears in portfolio management problems.	
\end{example}

In this example, we compare the numerical performance of Alg1, $\mbox{Alg1}_{\epsilon=0}$, the IHT algorithm proposed by Lu \cite{Lu2014Iterative} and the accelerated hard thresholding (AHT) algorithm proposed in \cite{Bian2023Accelerated} for solving problem \eqref{cs}.
The Lipschitz constant of $\nabla f$ and the inital point are chosen as $L_f=\|A^{\mathrm{T}}A\|$ and $x^0=\mathrm{ones}(n,1)$, respectively. For simplicity, we denote $c=\frac{1}{h}+(2\beta+h)L_f$. Some parameters of the algorithms are set as follows:
$$\beta=0.01,\ h=10,\ \epsilon=1\mathrm{e}-3,\ \bar{\epsilon}=2\epsilon,\ \gamma_{\tiny{\mbox{Alg1}}}=c,\ \gamma_{\tiny{\mbox{Alg1}_{\epsilon=0}}}=2c.$$
The parameter $L$ in the IHT and AHT algorithms is chosen to be $L=6L_f$.
For any fixed problem $(m,n,\mbox{Spar})$, we randomly generate the problem data as follows:
\begin{center}
	\tt{\text{$A=\mbox{np.random.randn}(m,n);\ s=\mbox{Spar}*n;$}}\\
	\tt{\text{$x^{\star}=\mbox{np.random.uniform}(-2,6,(n,1));\ x^\star[:n-s]=0;$}}\\
	\tt{\text{$x^{\star}[x^{\star}>5]=5;\ x^{\star}[x^{\star}<-1]=-1;\ \mbox{np.random.shuffle}(x^\star);$}}\\
	\tt{\text{$\mbox{per}= 0.05 *\mbox{np.random.randn}(m, 1);\
			b=A\mbox{.dot}(x^\star)+\mbox{per}.$}}\\
	
\end{center}

In Fig. \ref{fig:fig5.2} and Fig. \ref{fig31.a}, we present the average results of $100$ independent trials for problem \eqref{cs} with parameters $(m, n, \mathrm{Spar})=(2000, 400, 20\%)$. At this point, the termination condition is set to $300$ iterations. The figures illustrate the averaged outcomes across all trials, focusing on three key aspects: the convergence of $\{\|x^k\|_0\}$, the behavior of $\{F(x^k)-F^{\epsilon}\}$ and the changes in $\{\sum\|x^{k+1}-x^k\|\}$ during the iterative process. Here, $F^{\epsilon}$ denotes the minimum of the objective function values produced by four algorithms.

\begin{figure}[htbp]
	\centerline{
		\subfigure[$\{\|x^k\|_0\}$]{\label{fig2.a}\includegraphics[width=0.55\textwidth]{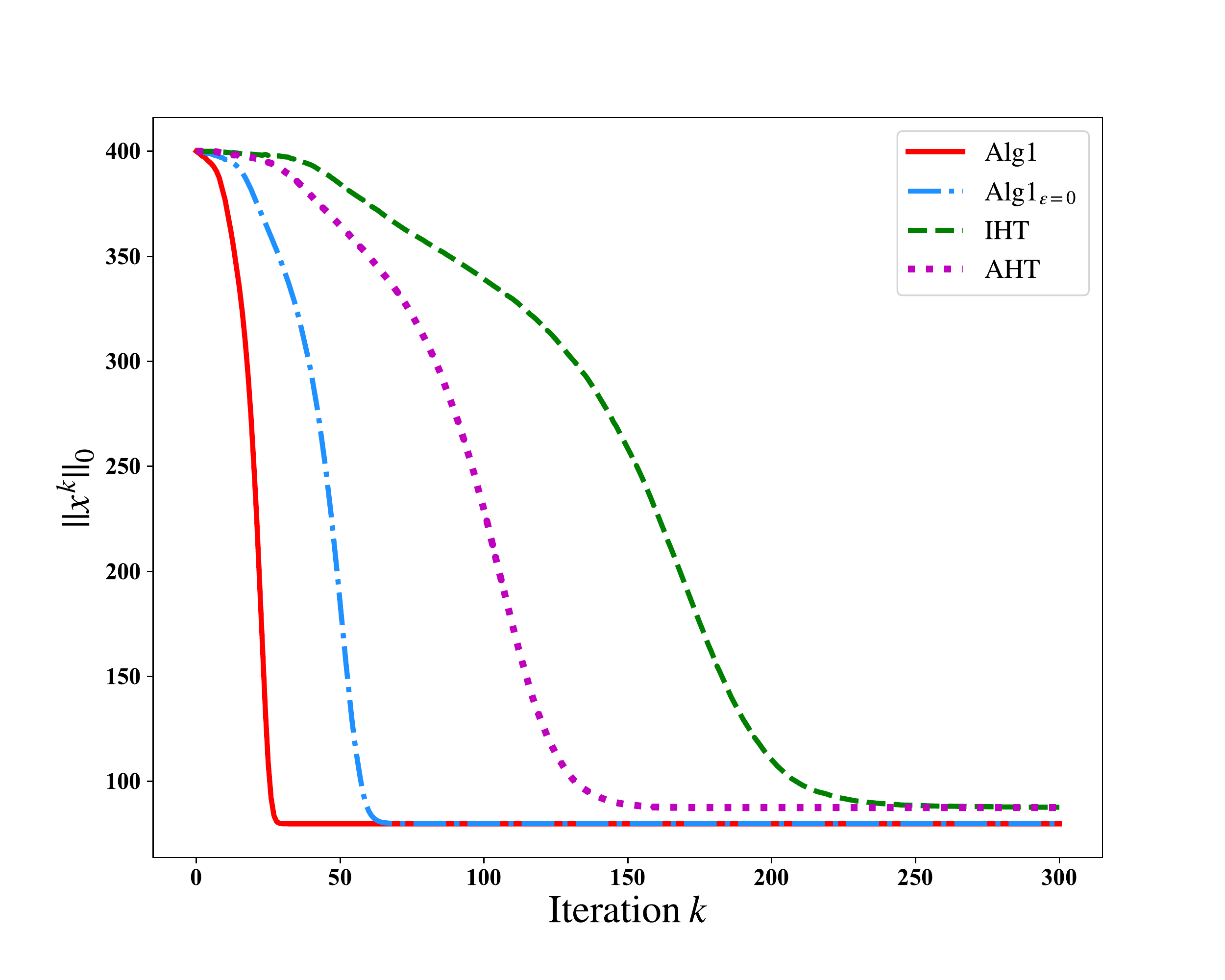}}
		\subfigure[$\{F(x^k)-F^{\epsilon}\}$]{\label{fig2.b}\includegraphics[width=0.55\textwidth]{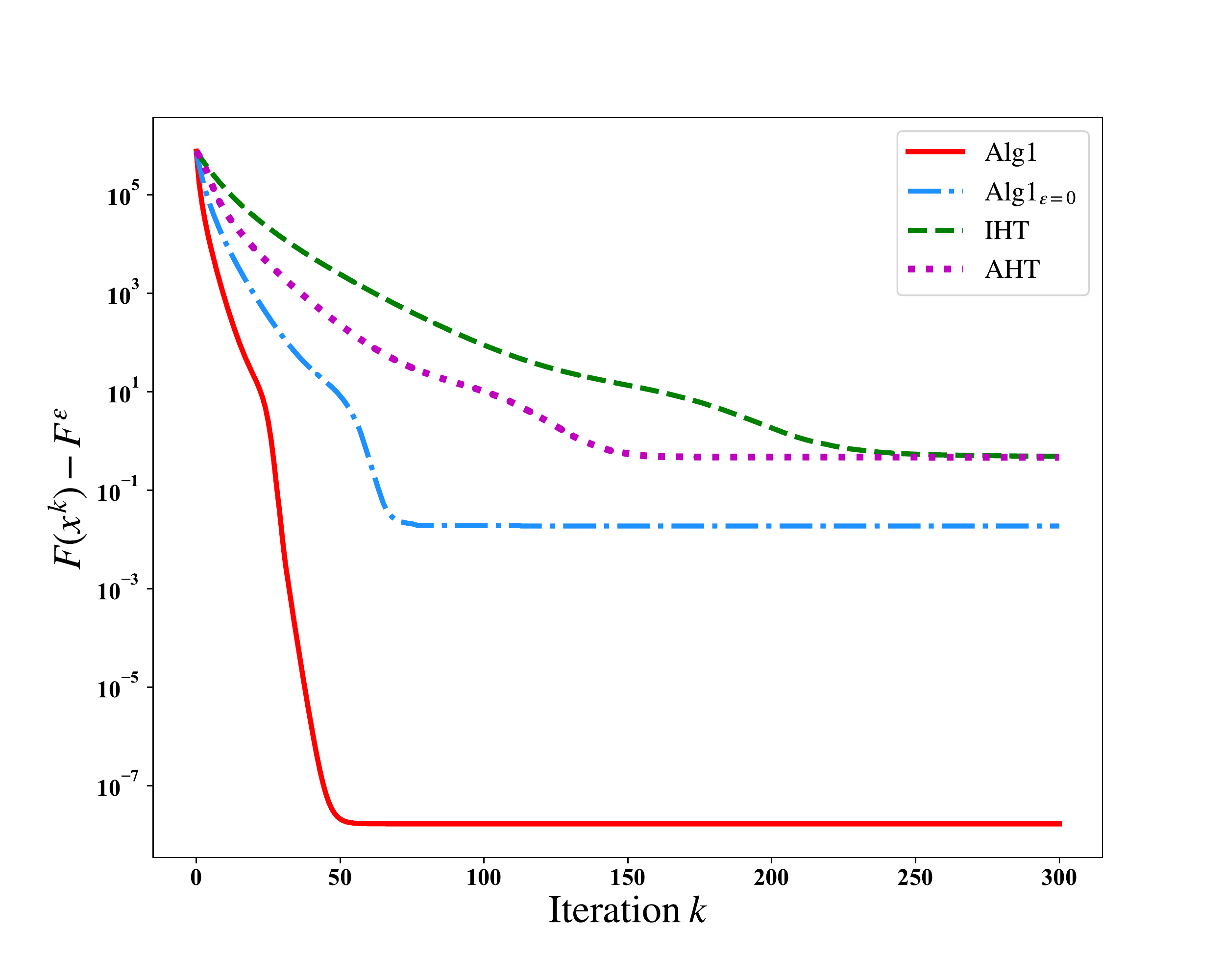}}
	}
	\caption{The averages of numerical results for Example \ref{ex2} with $m=2000$, $n=400$ and $\mbox{Spar}=20\%$.}
	\label{fig:fig5.2}      
\end{figure}

Fig. \ref {fig:fig5.2} demonstrates that Alg1 reaches an $\bar{\epsilon}$-local minimizer more quickly than the other algorithms, followed by $\mbox{Alg1}_{\epsilon=0}$ and the AHT algorithm, with the IHT algorithm being the slowest. These results confirm that incorporating extrapolation significantly accelerates the convergence rate compared to algorithms without extrapolation. In Fig. \ref{fig2.a}, it is evident that the zero norms of the solutions obtained by Alg1 and $\mbox{Alg1}_{\epsilon=0}$ are closer to that of the true solution, indicating higher accuracy. From the perspective of the cumulative changes $\{\sum\|x^{k+1}-x^k\|\}$, Fig. \ref{fig31.a} also shows that Alg1 and $\mbox{Alg1}_{\epsilon=0}$ exhibit a faster convergence rate compared to the other methods.  

\begin{figure}[htbp]
	\centerline{
		\subfigure[ $(m, n, \mbox{Spar})=(2000, 400, 20\%)$]{\label{fig31.a}\includegraphics[width=0.55\textwidth]{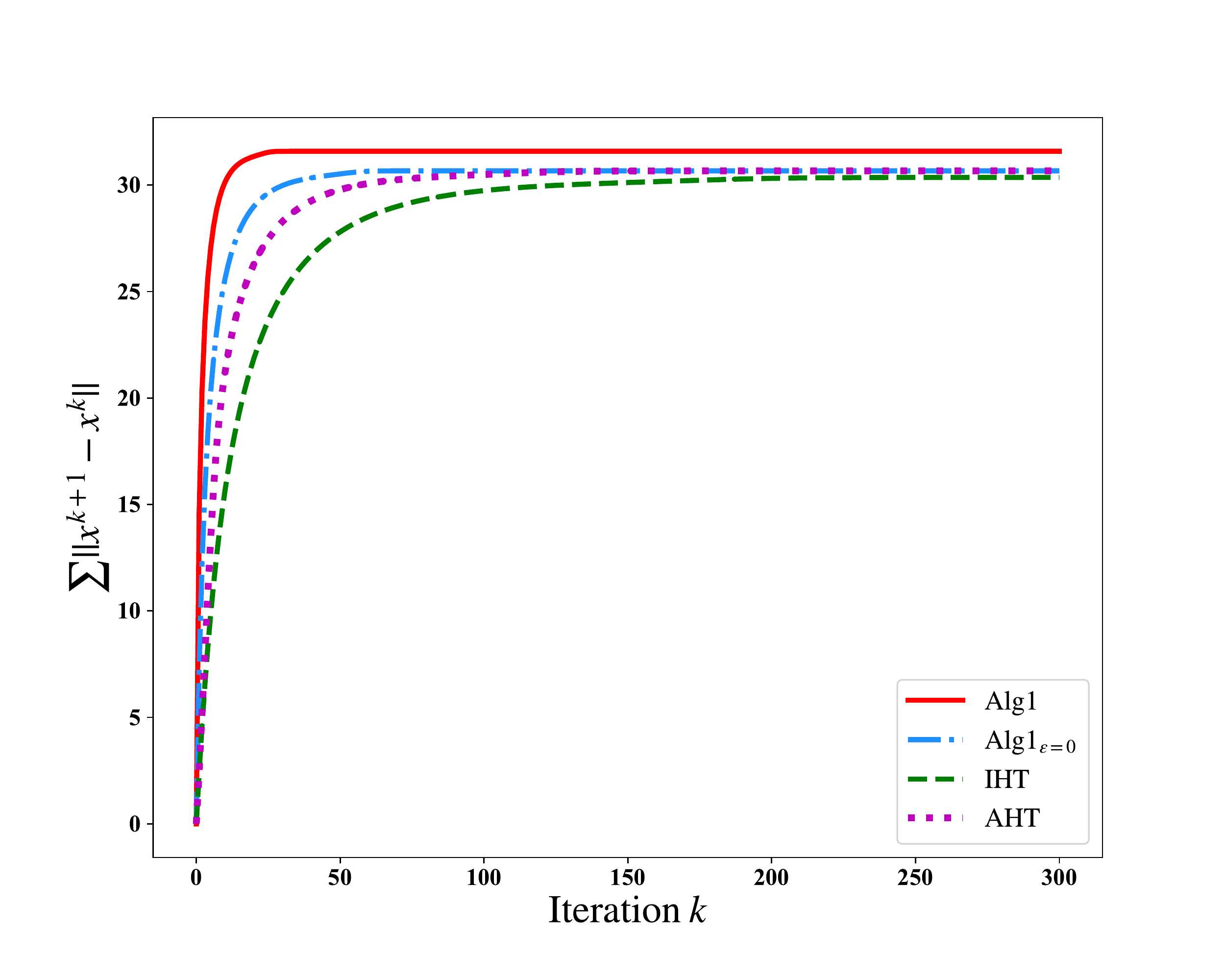}}
		\subfigure[ $(m, n, \mbox{Spar})=(3000, 600, 10\%)$]{\label{fig31.b}\includegraphics[width=0.55\textwidth]{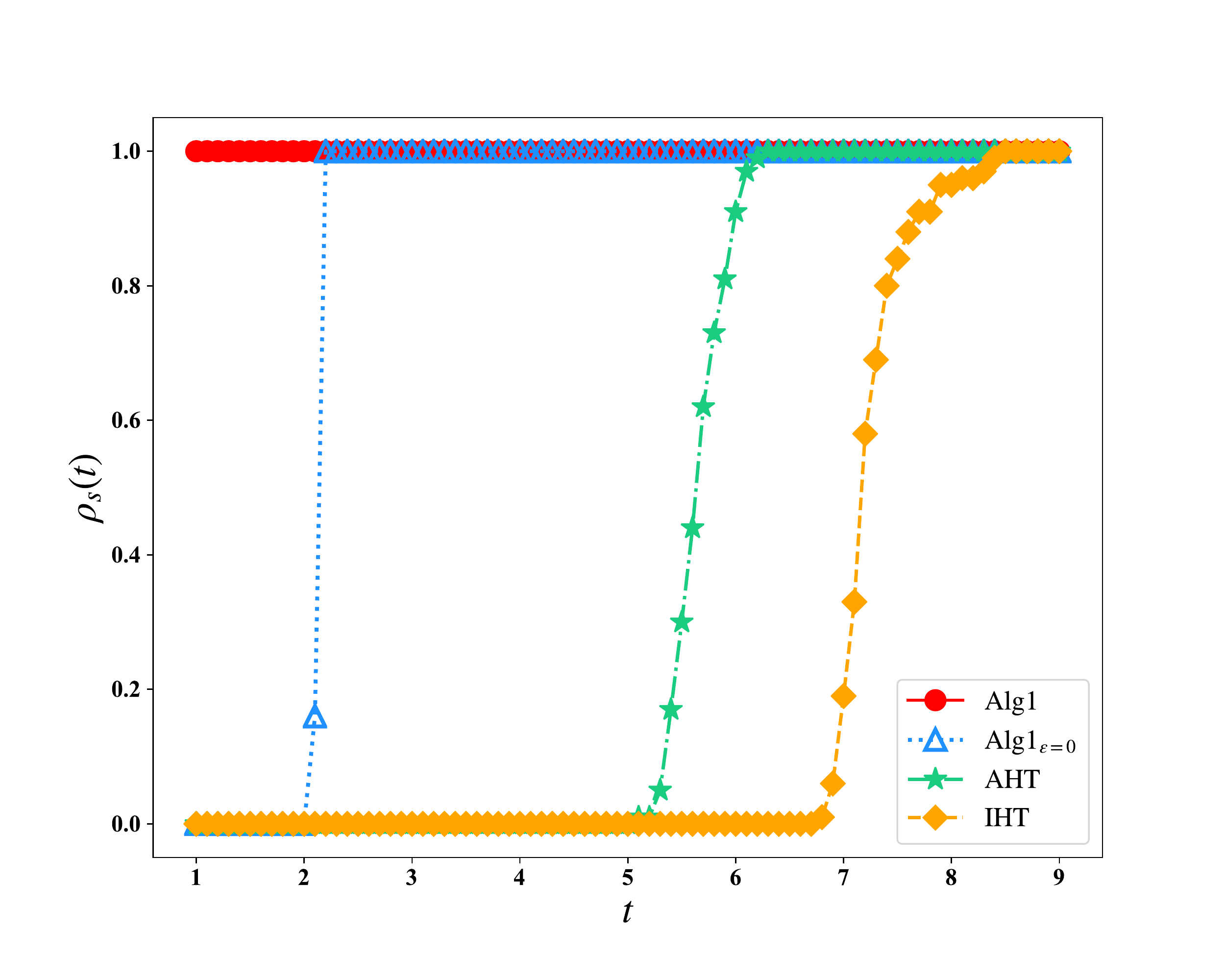}}
	}
	\caption{Convergence of $\{\sum\|x^{k+1}-x^k\|\}$ and the performance profiles of algorithms for Example \ref{ex2}.}
	\label{fig:fig5.31}       
\end{figure}

\begin{figure}[htbp]
	\centerline{
		\subfigure[ The relative error]{\label{fig3.a}\includegraphics[width=0.55\textwidth]{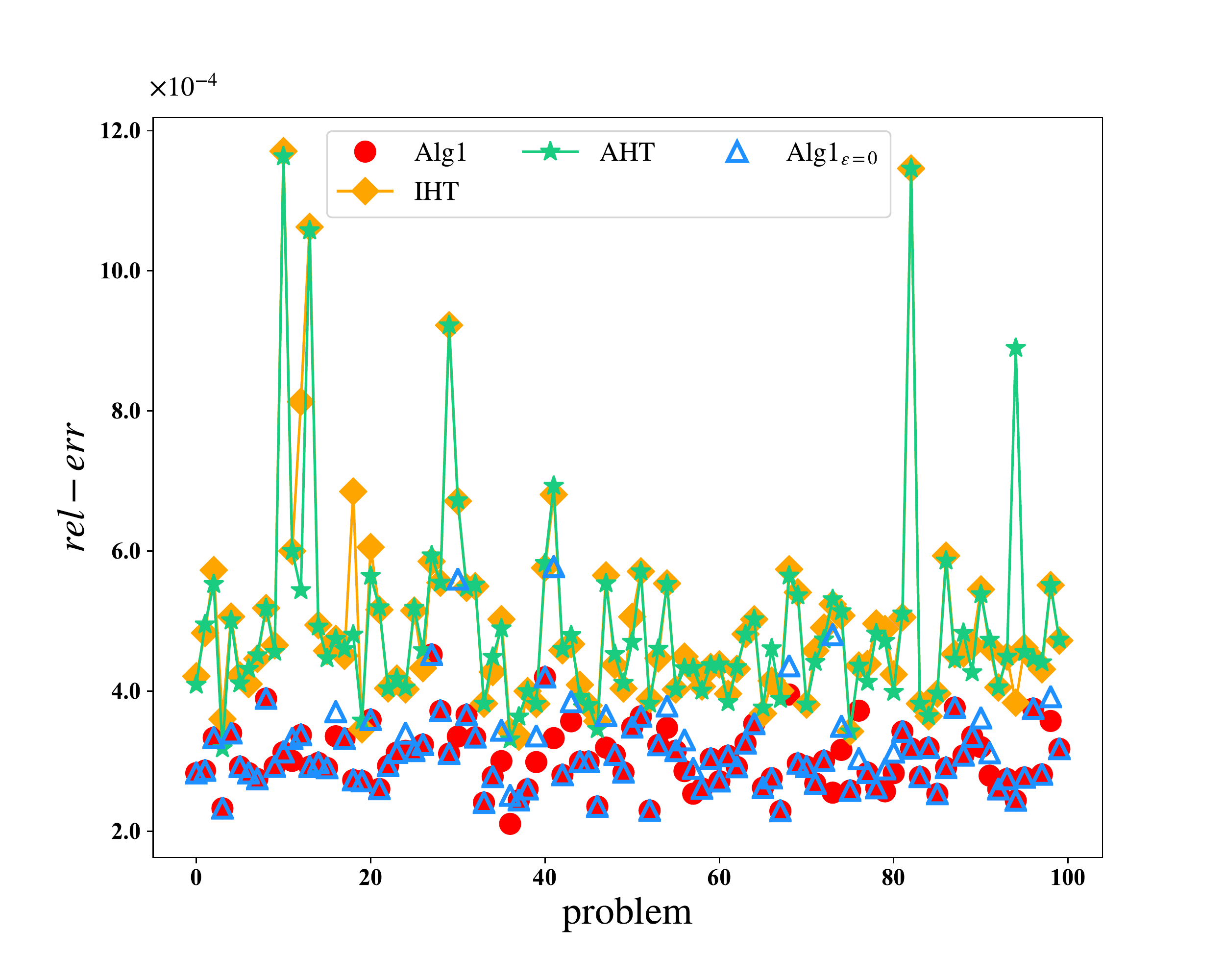}}
		\subfigure[ The sparsity regression rate]{\label{fig3.b}\includegraphics[width=0.55\textwidth]{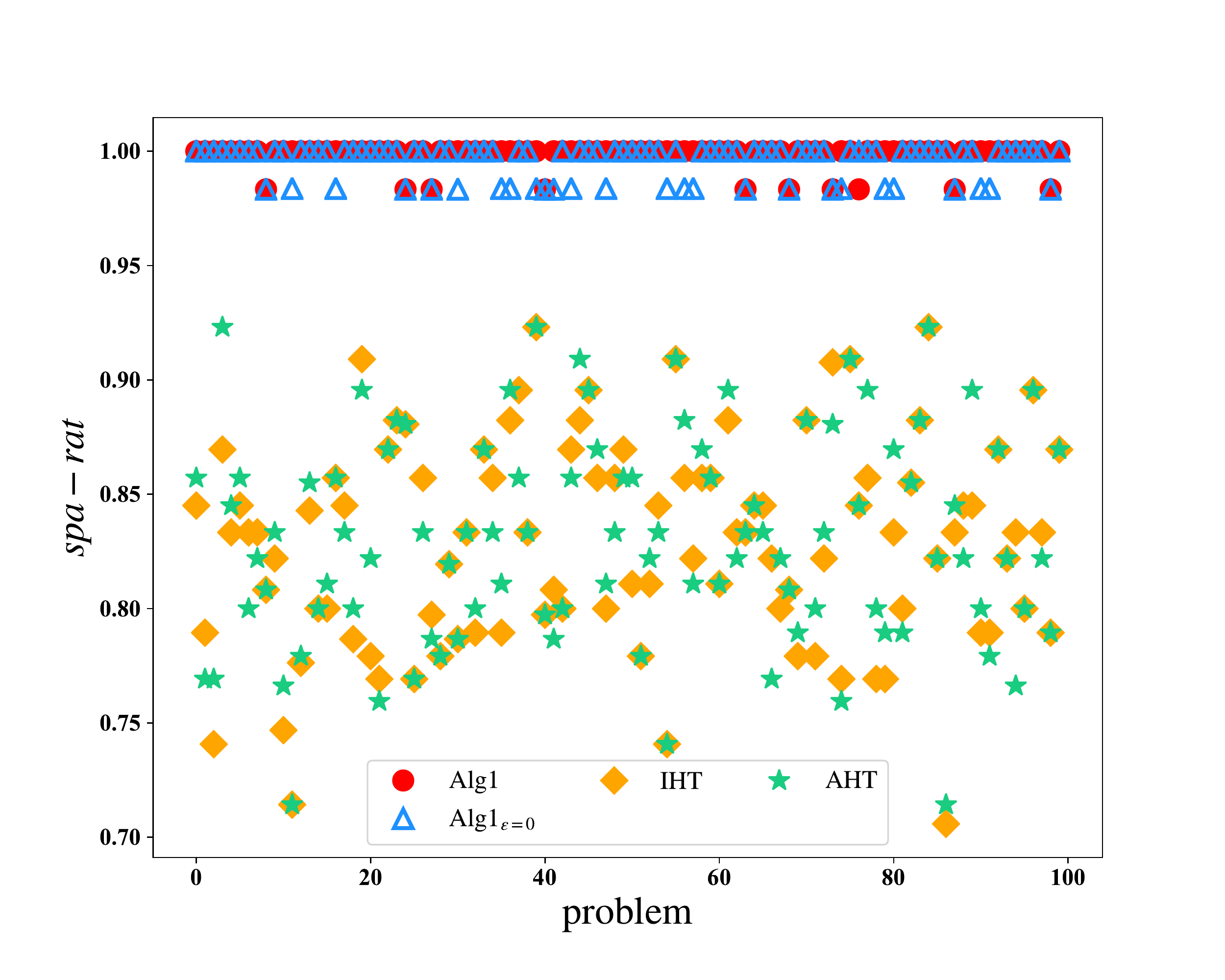}}
	}
	\caption{The relative errors and sparsity regression rates for Example \ref{ex2} with $m=3000$, $n=600$ and $\mbox{Spar}=10\%$.}
	\label{fig:fig5.3}       
\end{figure}

When $(m, n, \mbox{Spar})=(3000, 600, 10\%)$ is applied, we conduct $100$ independent experiments and present the results in Fig. \ref{fig31.b} and Fig. \ref{fig:fig5.3}. 
The performance profile \cite{Dolan2002Benchmarking}, relative error (\textit{rel-err}) and sparsity regression rate (\textit{spa-rat}) are chosen as the tools for comparing the solvers. Let $\mathcal{S}$ be the set of algorithms and $\mathcal{P}$ be the set of problems. We define the performance ratio $r_{p, s}=\frac{k_{p, s}}{\min\{k_{p, s} : s\in\mathcal{S}\}}$, where $k_{p, s}$ is the number of iterations required to solve problem $p\in\mathcal{P}$ by algorithm $s\in\mathcal{S}$. The performance profile of the algorithm $s\in\mathcal{S}$ is defined by
\begin{equation*}
\rho_{s}(t)=\frac{\left|\{p\in\mathcal{P}: r_{p, s}\leq t\}\right|}{|\mathcal{P}|},
\end{equation*}
where $t\in\mathbb{R}$. Next, we introduce the definitions of the relative error and sparsity regression rate of the final output point $x^a$ with respect to $x^*$, i.e.
\begin{equation*}
\mbox{\textit{rel-err}}=\frac{\|x^a-x^*\|}{\|x^a\|}\quad\mbox{and}\quad \mbox{\textit{spa-rat}}=\frac{|\Gamma(x^a)\cap\Gamma(x^*)|}{\max\{|\Gamma(x^a)|, |\Gamma(x^*)|\}}.
\end{equation*}

Fig. \ref{fig31.b} further illustrates that, compared to all the considered algorithms, Alg1 requires the fewest iterations to solve the problem, while $\mbox{Alg1}_{\epsilon=0}$ follows with a slightly higher count. Additionally, the results shown in Fig. \ref{fig:fig5.3} indicate that Alg1 and its variant $\mbox{Alg1}_{\epsilon=0}$ have a higher probability of successfully identifying the true solution, and outperform other methods in terms of both accuracy and speed. In particular, they exhibit faster convergence and lower misidentification rates when determining the positions of the nonzero and zero elements of the true solution.

\vspace{0.12cm}

\begin{example}\label{ex3}
	In this example, we solve the sparse logistic regression problem by Alg1, $\mbox{Alg1}_{\epsilon=0}$, IHT and AHT, which has numerous applications in medical diagnois and bioinformatics. For $m$ samples $\{a_i\}_{i=1}^m$ with $n$ features, we denote $x=(w, v)\in\mathbb{R}^{n+1}$ and the specific form of the problem is as follows
	\begin{equation}\label{lg3}
	\min_{\mathbf{-1}\leq x\leq \mathbf{1}} \frac{1}{m}\sum_{i=1}^m \log (1+e^{-y_i(d_ix)})+0.05\|w_+\|_0,
	\end{equation}
	where $y_i\in\{-1, 1\}$ and $d_i=(a_i^\mathrm{T}, 1)$ is the $i$th row of matrix $D\in\mathbb{R}^{m\times (n+1)}$.
\end{example}

For the given parameter $(m,n,\mbox{Spar})$, the data is randomly generated by
\begin{center}
	\tt{\text{$A=\mbox{np.random.randn}(m,n);\ s=\mbox{Spar}*n;$}}\\
	\tt{\text{$x^{\star}=\mbox{np.random.uniform}(0,2,(n,1));\ x^\star[:n-s]=0;$}}\\
	\tt{\text{$x^{\star}[x^{\star}>1]=1;\ \mbox{np.random.shuffle}(x^\star);\ v = \mbox{np.random.rand}(1)[0];$}}\\
	\tt{\text{$y = \mbox{np.sign}(A\mbox{.dot}(x^\star)+ v);\ x^\star= \mbox{np.append}(x^\star, v).$}}\\
\end{center}
Let $x^a=(w^a,v^a)$ denote the obtained point when the algorithm terminates, and define $x^{ap}=(w_+^a,v^a)$, where $w_+^a:=\max\{w^a,0\}$.
We set $L_f=\frac{1}{m}\|A\|^2$, $x^0=0.5*\mathrm{ones}(n+1,1)$ and $L=9L_f$, while keeping the remaining parameters identical to those in Example \ref{ex2}.

For the problem \eqref{lg3} with $(m, n, \mbox{Spar})=(400, 50, 20\%)$, the numerical results are illustrated in Fig. \ref{fig:fig6.1}. It can be observed that the outputs $x^k$ of both Alg1 and $\mbox{Alg1}_{\epsilon=0}$ not only satisfy the termination condition faster but also are closer to the true solution.

\begin{figure}[htbp]
	\centerline{
		\subfigure[ $\{\|w^k_+\|_0\}$]{\label{fig3.a}\includegraphics[width=0.55\textwidth]{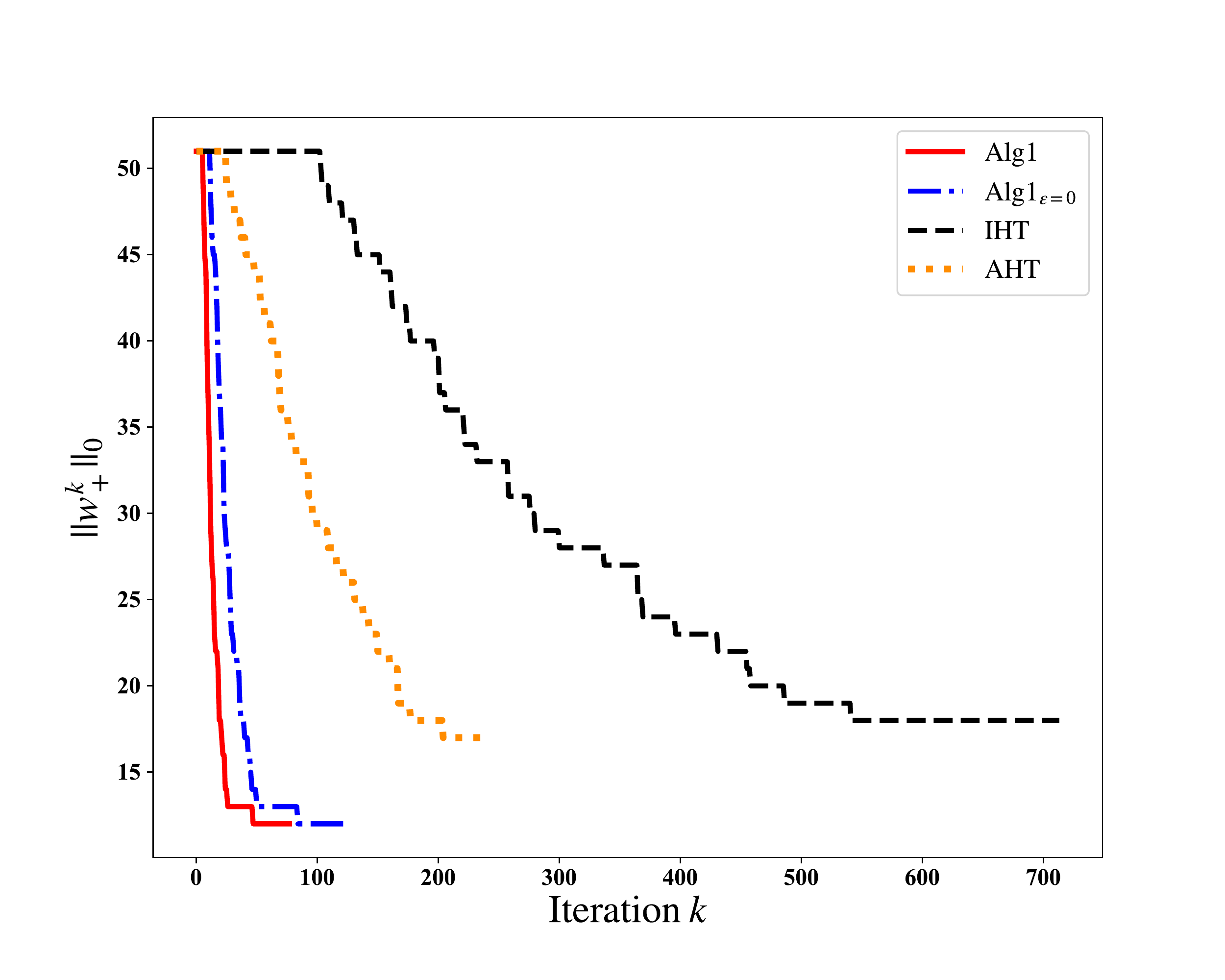}}
		\subfigure[ $\{F(x^k)\}$]{\label{fig3.b}\includegraphics[width=0.55\textwidth]{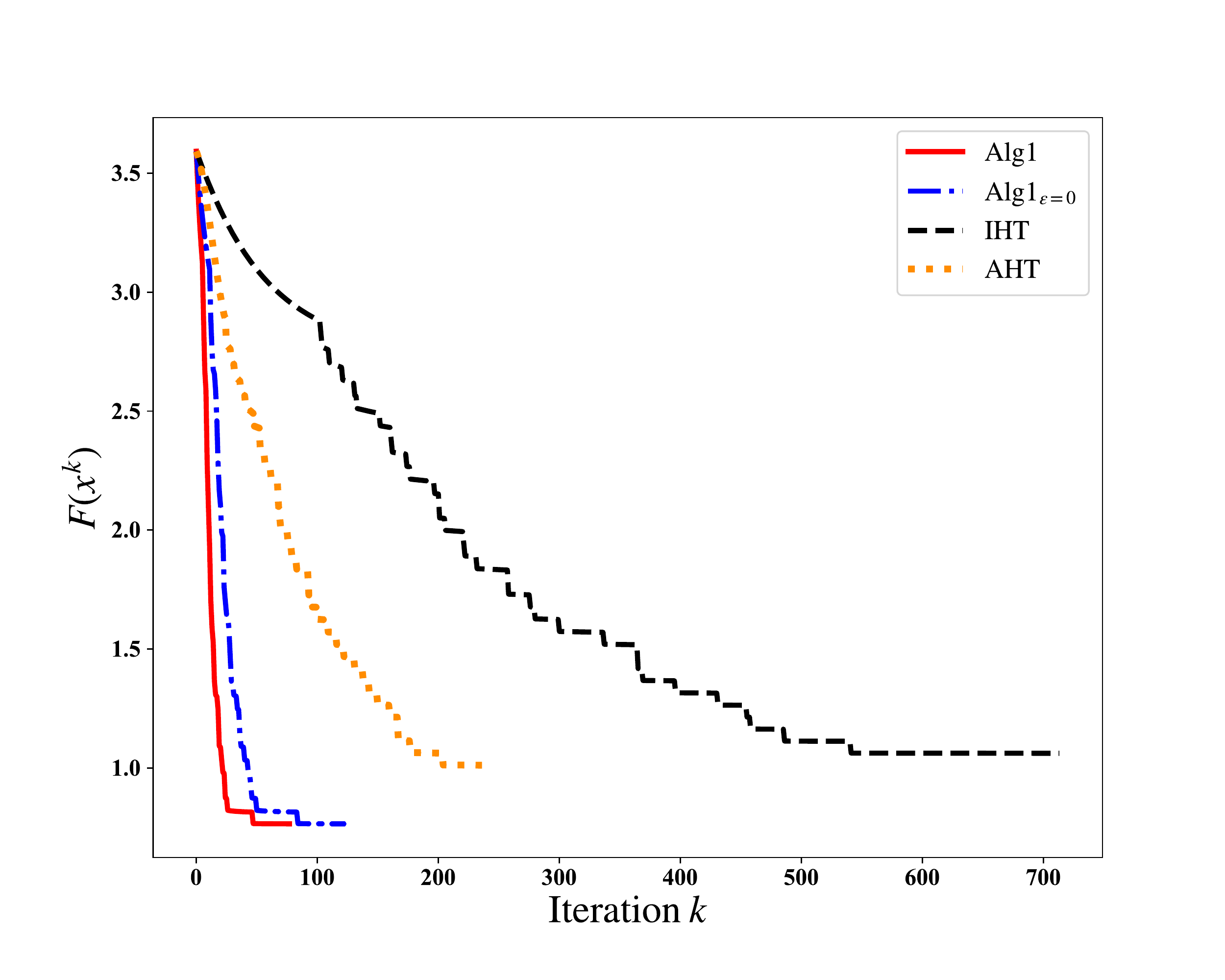}}
	}
	\caption{Convergence of $\{\|w^k_+\|_0\}$ and $\{F(x^k)\}$ for Example \ref{ex3} with $m=400$, $n=50$ and $\mbox{Spar}=20\%$.}
	\label{fig:fig6.1}     
\end{figure}

\begin{figure}[htbp]
	\centerline{
		\subfigure[Training\ problem]{\label{fig3.a}\includegraphics[width=0.55\textwidth]{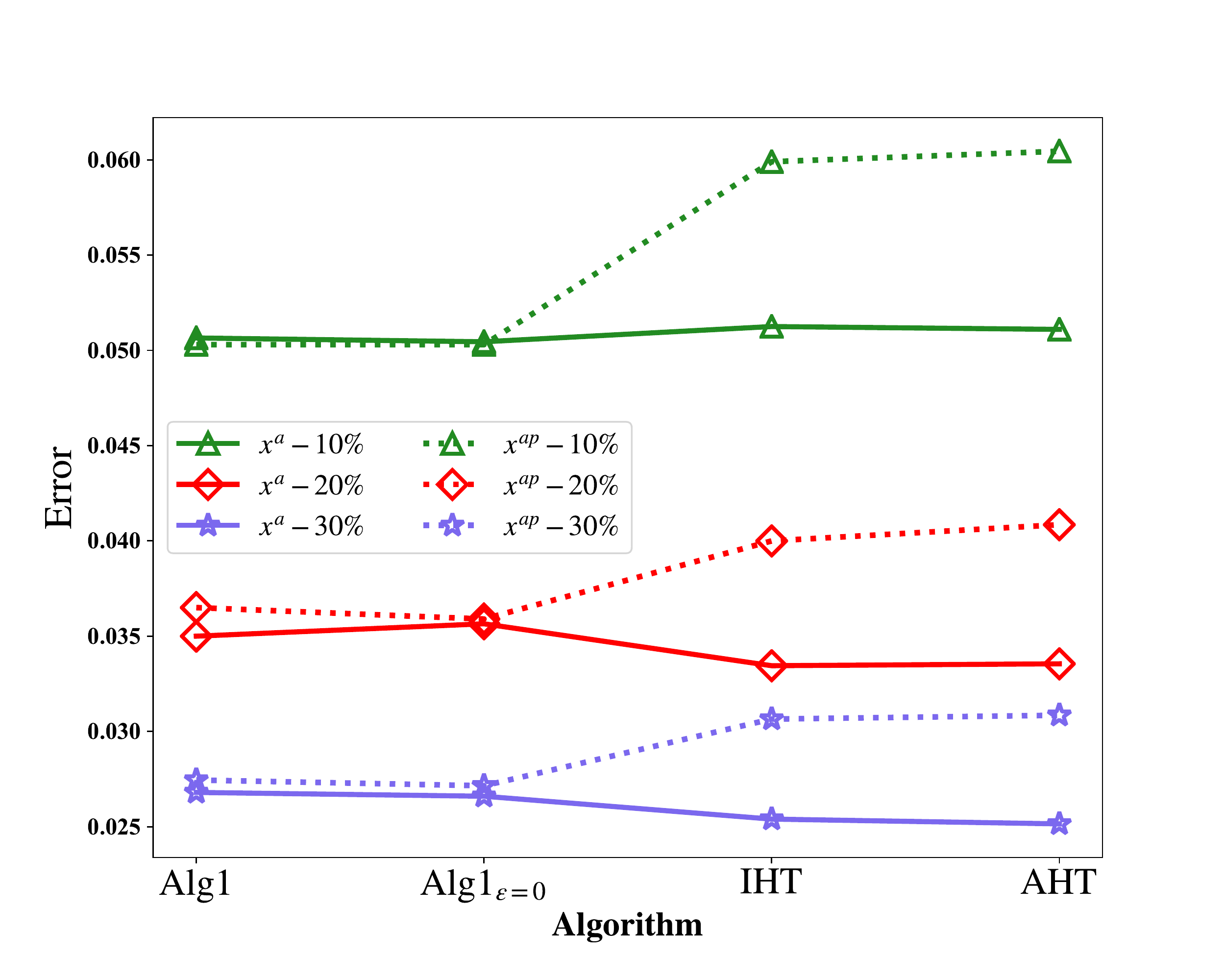}}
		\subfigure[Testing\ problem]{\label{fig3.b}\includegraphics[width=0.55\textwidth]{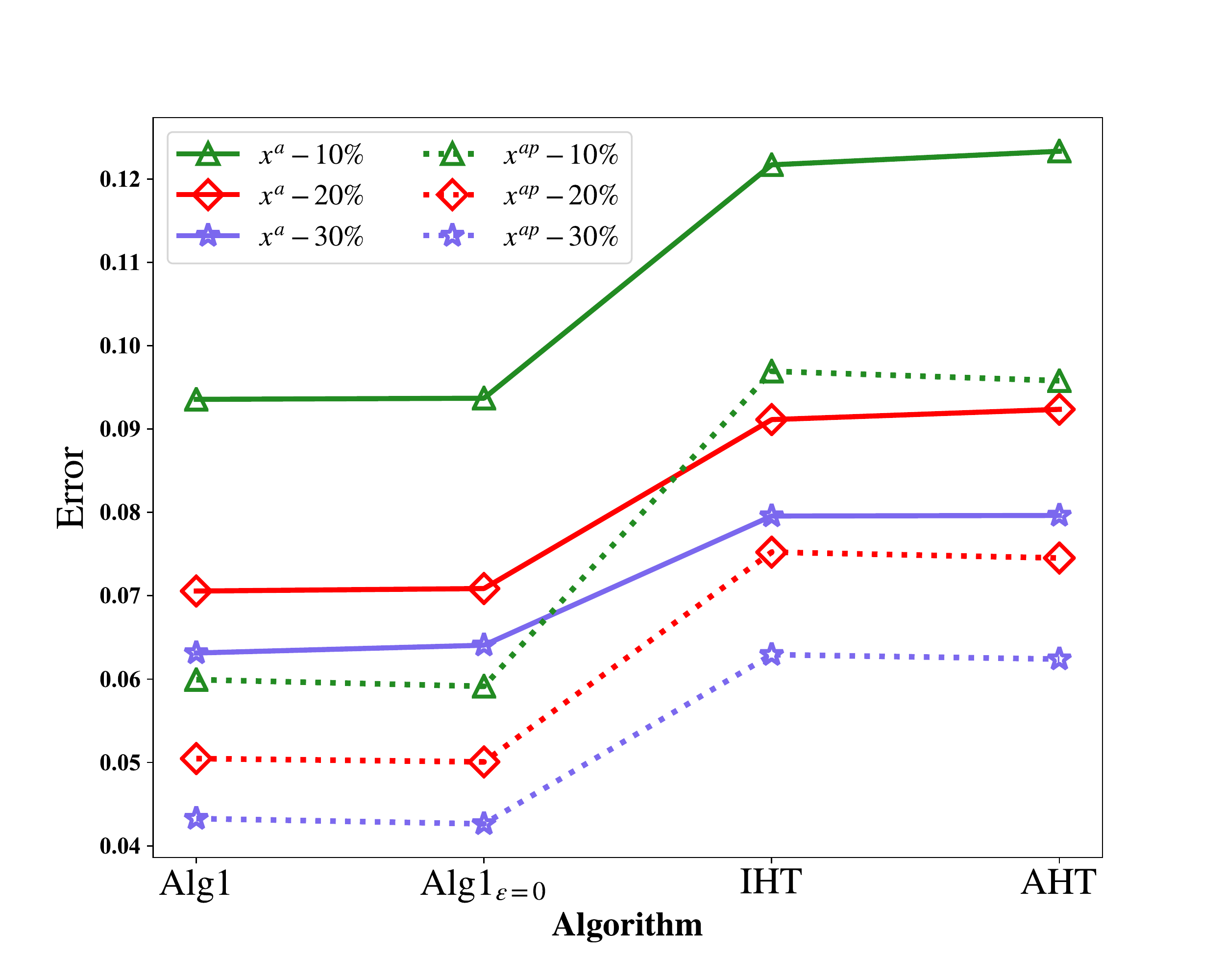}}
	}
	\caption{The error rate of each algorithm for Example \ref{ex3} with different sparsity levels.}
	\label{fig:fig6.4}       
\end{figure}

We randomly generate a dataset of $1000$ samples, each with $50$ features, and partition these samples into two subsets: $400$ samples for training and $600$ samples for testing. Given the model parameters  $x=(w,v)\in\mathbb{R}^{n+1}$ and a sample $a_i\in\mathbb{R}^n$, the predicted outcome $y_i\in\{-1, 1\}$ is determined by
\begin{equation*}
r(a_i)=\mbox{sgn}({a_i}^{\mathrm{T}}w+v),
\end{equation*}
where the function $\mbox{sgn}(t)$ outputs $1$ if $t>0$ and $-1$ otherwise. To evaluate the quality of the obtained solution, we compute the error rate \cite{Lu2012Sparse} as follows:
\begin{equation*}
\mbox{Error}:=\frac{1}{m}\cdot \sum_{i=1}^m \|r(a_i)-y^{true}_i\|_0,
\end{equation*}
where $m$ denotes the number of samples and $y^{true}_i$ represents the truth label of sample $a_i$.

To investigate the impact of sparsity, we conduct $50$ independent tests for different sparsity levels, generating fresh data for each test. The average error rates of various algorithms over these $50$ tests are displayed in Fig. \ref{fig:fig6.4}. As evident from the results, the weighting coefficients obtained by Alg1 and $\mbox{Alg1}_{\epsilon=0}$ yield relatively lower errors in classification tasks compared to other methods when tested on the problem.

\begin{figure}[htbp]
	\centerline{
		\subfigure[Iterations $k$ of each algorithm]{\label{fig3.a}\includegraphics[width=0.5\textwidth]{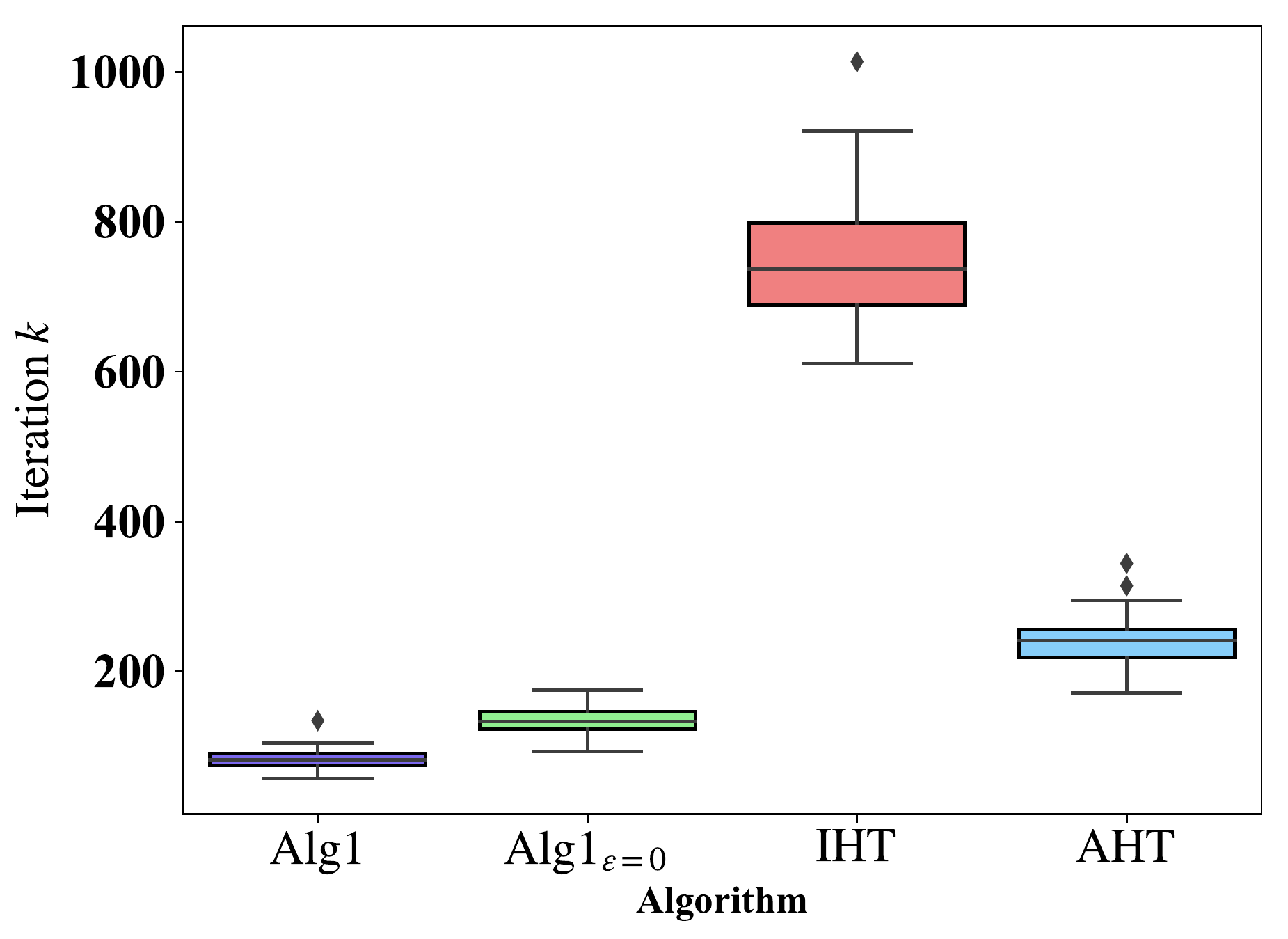}}
		\subfigure[$\|w^a_+\|_0$ of each algorithm]{\label{fig3.b}\includegraphics[width=0.5\textwidth]{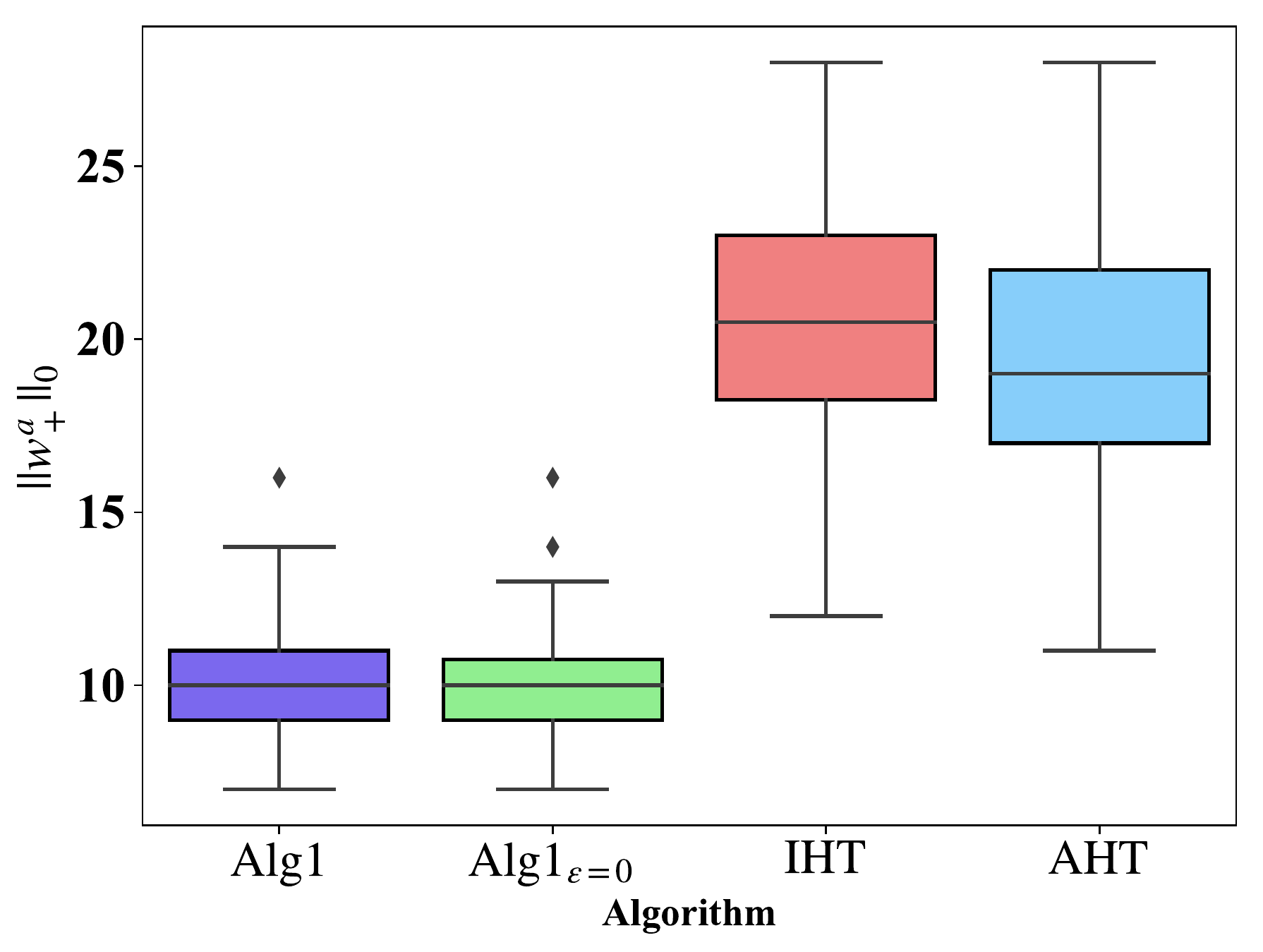}}
	}
	\caption{The number of iterations and $\|w^a_+\|_0$ for Example \ref{ex3} with $m=400$, $n=50$ and $\mbox{Spar}=20\%$.}
	\label{fig:fig6.3}       
\end{figure}

In Fig. \ref{fig:fig6.3}, we present a box plot to illustrate the number of iterations and $\|w^a_+\|_0$ from $50$ independent trials solved by different algorithms under a sparsity level of $20\%$. This visualization provides a clearer demonstration of the rapid convergence achieved by Alg1.

\section{Conclusion}\label{sec13}
In this paper, we propose an extrapolated hard thresholding algorithm with Hessian-driven damping and dry friction for solving problem \eqref{pro1}, by temporally discretizing an inertial gradient system with dry friction and Hessian-driven damping. When the dry friction coefficient $\epsilon\neq 0$, we prove that $\sum_{k=1}^{\infty}\|x^{k+1}-x^k\|<\infty$. This implies that the iterates generated by the proposed algorithm have finite length. Notably, our proof does not rely on the K{\L} property. Moreover, we provide an equivalent characterization of the local minimizers of problem \eqref{pro1} and based on this, define a class of $\epsilon$ approximate local minimizers. We further prove that the iterates converge to an $\epsilon$-local minimizer of the problem. In the case where $\epsilon=0$, while the algorithm does not exhibit the finite length property, we show that any accumulation point of the iterates is a local minimizer. We discuss the stability of the algorithm and give a sufficient condition on the errors to guarantee the same convergence results to the algorithm without perturbation.
Finally, we demonstrate through numerical simulations that the proposed algorithm can achieve a better solution more quickly.

\backmatter

\bmhead{Acknowledgements}
The research of Fan Wu is partially supported by the Postdoctoral Fellowship Program of CPSF under Grant Number GZC20233475. The research of Wei Bian is partially supported by the National Natural Science Foundation of China Grants (12425115, 12271127, 62176073).

\section*{Declarations}

\bmhead{Data Availibility}
The datasets generated during the current study are available from the corresponding author upon reasonable request.

\bmhead{Conflict of interest}
The authors declare no conflict of interest.

\bibliography{reference}
\end{document}